\documentclass[journal]{IEEEtran}

\usepackage[utf8]{inputenc}
\usepackage{graphicx}
\usepackage{subcaption}
\usepackage{multirow}
\usepackage{amsmath,amssymb,amsfonts}
\usepackage{mathrsfs}
\usepackage{xcolor}
\usepackage{textcomp}
\usepackage{manyfoot}
\usepackage{booktabs}
\usepackage{algpseudocode}
\usepackage{listings}
\usepackage{hyperref}
\usepackage{bm}
\usepackage{soul}
\usepackage{epstopdf}
\usepackage{caption}
\usepackage{comment}
\usepackage{algorithm}
\usepackage{cite}
\usepackage{placeins}

\usepackage[table]{xcolor}
\usepackage{tabularx}
\usepackage{tcolorbox}
\usepackage{tikz}
\usepackage{listings}
\usetikzlibrary{shapes.geometric, arrows.meta, positioning, backgrounds}
\usepackage{gensymb}
\pgfmathsetseed{\number\pdfrandomseed}

\newtheorem{definition}{Definition}

\newtheorem{proof}{Proof}
\newtheorem{proposition}{Proposition}
\newtheorem{theorem}{Theorem}
\newtheorem{remark}{Remark}

\newcommand{\halfquad}{\mspace{9mu}}
\newcommand{\sixteenquad}{\mspace{216mu}}

\definecolor{purple}{rgb}{0.6, 0.2, 0.8}
\definecolor{darkred}{rgb}{0.545, 0.0, 0.0}
\definecolor{darkgreen}{rgb}{0.0, 0.392, 0.0}
\definecolor{lightblue}{rgb}{0.68, 0.85, 0.9}

\title{Reliability-Aware Control of Distributed Energy Resources using Multi-Source Data Models}
\author{Gejia Zhang, Robert Mieth}

\begin{document}

\bstctlcite{IEEE:BSTcontrol} 

\maketitle

\begin{abstract}
    Distributed energy resources offer a control-based option to improve distribution system reliability by ensuring system states that positively impact component failure rates.
    This option is an attractive complement to otherwise costly and lengthy physical infrastructure upgrades.
    However, required models that adequately map operational decisions and environmental conditions to system failure risk are lacking because of data unavailability and the fact that distribution system failures remain rare events.
    This paper addresses this gap and proposes a multi-source data model that consistently maps comprehensive weather and system state information to component failure rates.
    To manage collinearity in the available features, we propose two ensemble tree-based models that systematically identify the most influential features and reduce the dataset's dimensionality based on each feature's impact on failure rate estimates. 
    These estimates are embedded within a sequential, non-convex optimization procedure, that dynamically updates operational control decisions.
    We perform a numerical experiment to demonstrate the cost and reliability benefits that can be achieved through this reliability-aware control approach and to analyze the properties of each proposed estimation model.
\end{abstract}

\section{Introduction}

Rising electricity demand, driven by population growth and electrification initiatives, has intensified concerns about the adequacy of electric power distribution systems \cite{lee2020does, emmanuel2017evolution}, and their reliability in the context of their environmental conditions \cite{panteli2015influence, hughes2021damage, luca2022reliability, hani2023weather}. 
This includes weather-related wear, such as damage from precipitation, but also safe loading limits depending on ambient temperature, thermal radiation, and wind conditions \cite{farzaneh2005statistical, tilman2012probabilistic, luo2024resilience, glaum2023leveraging}. 
Also, vegetation growth near distribution lines contributes to distribution outages, with adverse weather further increasing outage likelihood and severity \cite{doostan2020datadriven, jumbo2022resource}. 
Finally, weather patterns also drive consumption patterns and directly impact operating constraints which, in turn, define the system's safe operating space \cite{sukprasert2024implications, moein2016optimal}. In addition, even at fixed weather, the operating point itself alters component failure risk, since currents and voltages contribute to component stress. Reliability is therefore decision-dependent \cite{richard_epdr_3}.
Controllable distributed energy resources (DERs) offer a control-based option to positively contribute to distribution system reliability by ensuring system states that correspond to lower failure risks \cite{zhang2025towards}. 
While this option is a cost-effective complement to physical infrastructure upgrades, models to precisely quantify failure risk based on weather context \textit{and} operational decisions are lacking. This paper addresses this~gap.

Linking distribution system component reliability with the covariates that we consider is hindered by a spatial mismatch on the meteorological side,
because publicly available meteorological datasets typically provide data on a fixed spatial grid across broad regions that rarely coincide with the substation service area of interest \cite{copernicus}.
For example, the Copernicus reanalysis at $0.25^{\circ}$ resolution, provides model fields on grid cells of about $215.06\,\text{mi}^2$ at mid-latitudes (roughly $17.26\,\text{mi}$ north to south and $12.46\,\text{mi}$ east to west) \cite{copernicus}. This is a much coarser resolution than the typical extend of around $25\,\text{mi}^2$ of distribution feeders in a substation service area \cite{tom_epdh_1}. 
Weather drivers of distribution component failures, such as wind gusts, wind direction, and precipitation, vary over short distances, so a single nearest grid cell is a noisy surrogate \cite{van2015temporal, jonassen2012multi, wang2023causes, tervo2020predicting}. At these scales, it is therefore necessary to combine weather variables from multiple nearby grid points to form a standard local approximation at the substation rather than rely on a single weather grid point.

Existing literature that integrates weather conditions into power system reliability assessments typically emphasizes extreme events capable of causing catastrophic failures \cite{kim2018enhancing, wang2017robust, jufri2019state} and focuses predominantly on network reconfiguration planning measures or infrastructure hardening \cite{ma2016resilience, he2018robust, li2023distributionally, li2023robust, lin2018tri}. Multiple studies cast extreme event robustness as tri-level optimization model with a min-max-min structure. The upper level selects line hardening \cite{ma2016resilience, he2018robust, lin2018tri, li2023distributionally} and DER allocation as well as proactive reconfiguration \cite{li2023distributionally}. The middle level exposes the system to worst case disruptions, including extreme weather induced outages \cite{ma2016resilience, li2023distributionally}, natural disaster \cite{he2018robust}, adversarial attacks \cite{lin2018tri}, and low probability high damage contingencies in a multi energy setting \cite{li2023robust}. The lower level then minimize load shedding and operating cost. 

In contrast to these planning and resilience-focused models, few studies explicitly incorporate distribution system reliability into operational decision-making by linking component failure rates to control setpoints for DERs. This gap partly stems from the modeling and computational challenges of embedding weather-dependent reliability within optimization frameworks \cite{zhang2025towards}. 
To date, progress toward closing this gap remains limited. Our own previous work in  \cite{zhang2025towards} incorporates distribution system reliability into operational decision-making by co-optimizing operating cost and expected energy-not-served (EENS), but relies solely on temperature as the meteorological input. This limitation motivates our study to move beyond single-variable inputs and to incorporate richer meteorological information when predicting component failures. 

From a statistical modeling perspective, moving beyond temperature-only inputs to include meteorological variables from several nearby grid points places the task in the realm of multi-source data fusion, with each measurement site treated as a distinct information source \cite{yifan2024multi, xie2024enhancing}. 
Unstructured pooling of all sites into a single feature set yields strong within-site and cross-site correlations. The multicollinearity is structural and persists whenever all sources are used. Our goal is to develop prediction models that use the data efficiently, and are not susceptible to multicollinearity and demonstrate fast computational speed for multi-sourced multi-featured data. 

In our problem setup, we know the coordinates of each weather source and the full set of variables reported at each source. It is therefore advantageous to encode this prior knowledge by assigning soft weights to sites and to individual variables that reflect their expected relevance. 
In generalized linear models, multicollinearity manifests as an ill-conditioned information matrix and inflated standard errors, and in logistic regression it limits expressiveness to linear relations in the log-odds \cite{mason1991collinearity, obrien2007caution, king2001logistic, gelman_daurmhm_3, gelman_daurmhm_5}. Methods such as support vector machines (SVM) and neural networks can represent non-linear relations, but encoding grouped sources and prior weights requires careful feature scaling \cite{rakotomamonjy2008simplemkl, goodfellow_dp, trevor_teosl_11, trevor_teosl_12}. 

From a modeling perspective, tree-based methods address several of these issues. A tree-based model (e.g. CART style tree) chooses one predictor and threshold at each split, so each decision depends on the order statistics of a single variable, which is comparatively robust to collinear inputs. 
Tree ensembles also capture nonlinearities, and they allow weighted sampling over data sources and variables so that prior knowledge can be used without requiring exact weights \cite{breiman_cart_2, breiman1996bagging, leo2001randomforest, dumitrescu2022machine}. 
Furthermore, in practice, tree-based ensemble methods such as random forests are often faster to train than non-linear kernel SVM or neural networks. 
This is because tree induction scales more favorably and avoids expensive kernel computations and gradient-based optimization \cite{goodfellow_dp, trevor_teosl_15, trevor_teosl_16, xu2010simple}. 

Empirically, many studies report that weighting features by predictive strength improves tree ensemble performance \cite{zhang2024multi, gyeong2024double, liu2017variable, nguyen2015unbiased, maudes2012random}. Theoretical support, however, is more limited. The first non-asymptotic mean squared error (MSE) decomposition for random forests (RF) with adaptation to sparsity appears in \cite{biau2012analysis}. That decomposition is stated for an arbitrary fixed feature sampling distribution and does not establish whether reweighting improves performance relative to random sampling. 
Subsequent work provides finite sample bounds for the probability of splitting on strong variables and for single tree strength, but not a forest-level MSE bound \cite{borup2023targeting}. 
To the best of our knowledge, the analysis in this paper is the first to insert an explicit weight vector and to prove, in a non-asymptotic manner and for every sample size, that targeted sampling tightens the MSE bound relative to random sampling. 

From the data-level perspective, a complementary remedy of multicollinearity is principal component analysis (PCA), which replaces correlated features with orthogonal components and stabilizes estimation, at the cost of reduced interpretability \cite{yifan2024leastangle, yifan2024true, tianhui2024dynamic}. Sparse PCA improves interpretability via many zero loadings yet is unsupervised and prioritizes explained variance over association with the target \cite{wang2024sequentially, xie2024enhancing}. Because our features are naturally grouped by measurement site and by physical variable class, we adopt group-wise PCA within a boosting framework as secondary compression while preserving meaning at the group level, and we feed the resulting factors to the weighted tree base learner. 

In this work we co-optimize operating cost and risk from component failure measured by EENS, with failure rates predicted from data of an arbitrary number of nearby weather sources and feature sets. 
To address structural multicollinearity in this setting, we adopt a problem informed design that uses soft weights on sources and variables to reflect expected relevance and that favors algorithms with low tuning and runtime burden. 
Because no single predictive model is uniformly best for multi-source, multi-dimensional data with structural multicollinearity \cite{fernandez2014we, wolpert1997no}, we instantiate two predictors tailored to this design: A weighted multi-source decision tree ensemble and a boosting scheme that applies prior guided grouping of features with PCA within groups to reduce collinearity while preserving group level interpretability. Finally, we embed the resulting failure probabilities in a sequential optimization problem that iteratively updates component reliability estimates and DER control setpoints, thereby operationalizing weather- and decision-dependent reliability in a high dimensional non-convex setting.

\section{Tree-based Ensemble Model}
\label{sec:tree_ensemble}

Our goal is to learn a mapping from weather and power flow features to system component failure probability that $(i)$~accounts for information across measurement sites and $(ii)$~remains stable in the presence of cross-feature correlations that can cause multicollinearity and overfitting. Section \ref{subsubsec:weighted_feature_prop} presents a formal discussion on the benefits of statistically embedding prior knowledge about the features into the ensemble model. Sections \ref{subsubsec:ensemble_model_wmsdt} and \ref{subsubsec:ensemble_model_tsmsb} then introduce two suitable ensemble models.

We assume that standardized weather data from $D$ measurement sites, each providing $W$ features, is available. Standardization is performed separately for each feature at each measurement site by subtracting each feature's sample mean and dividing by its sample standard deviation, resulting in data with zero mean and unit variance. All sites share the same $W$ features, which we organize into $M$ categories based on domain knowledge. For instance, features with similar physical or measurement attributes, such as temperature-related variables, are grouped together. Let
\begin{equation*}
    \mathcal{D} = \{1, \cdots, D\}, \quad \mathcal{W} = \{ 1, \cdots, W \}, \quad \mathcal{M} = \{ 1, \cdots, M \}.
\end{equation*}
Each feature belongs to exactly one category and the subsets $\mathcal{W}_m \subseteq \mathcal{W}$ (for $m \in \mathcal{M}$) form a partition of $\mathcal{W}$. In addition to weather features, we also include power flow-based information. Define 
$\mathcal{O}^{\text{b}}$ and $\mathcal{O}^{\text{l}}$ as the sets of features corresponding to net power load for buses and current magnitudes for lines, respectively. Let $\Omega$ be the sample space of all possible observations and for each site $d \in \mathcal{D}$, category $m \in \mathcal{M}$, feature $w \in \mathcal{W}_m$, and $o \in \mathcal{O}^{\text{b}} \cup \mathcal{O}^{\text{l}}$ define the real-valued random variable (r.v.)
\begin{equation*}
    X_{d, m, w, o}: \Omega \rightarrow \mathbb{R}.
\end{equation*}
For the operational status of buses, we collect $n_{\text{b}}$ i.i.d. observations $\left\{ \omega_i \right\}_{i = 1}^{n_{\text{b}}} \subseteq \Omega$. The realized dataset is then
\begin{align*}
    \big\{ x_{i, d, m, w, o} =  X_{d, m, w, o}\left( \omega_i \right) & \mid i = 1, \cdots, n_{\text{b}}, \\
    d & \in \mathcal{D}, m \in \mathcal{M}, w \in \mathcal{W}_{m}, o \in \mathcal{O}^{\text{b}} \big\}.
\end{align*}
Similarly for the operational status of lines, we collect $n_{\text{l}}$ i.i.d. observations $\left\{ \omega_i \right\}_{i = 1}^{n_{\text{l}}} \subseteq \Omega$. The realized dataset is then
\begin{align*}
    \big\{ x_{i, d, m, w, o} =  X_{d, m, w, o}\left( \omega_i \right) & \mid i = 1, \cdots, n_{\text{l}}, \\
    d & \in\mathcal{D}, m \in \mathcal{M}, w \in \mathcal{W}_{m}, o \in \mathcal{O}^{\text{l}} \big\}.
\end{align*}
For simplicity and \textit{w.l.o.g.}, we will no longer distinguish bus and line observations. We represent the realized dataset for $n$ i.i.d. observations as
\begin{equation*}
    \big\{ x_{i, d, m, w} \!\! = \! \! X_{d, m, w} \! \left( \omega_i \right) \! \mid \! i \! = \! 1, \cdots, n, d \! \in \! \mathcal{D}, m \! \in \! \mathcal{M}, w \! \in \! \mathcal{W}_{m} \! \big\},
\end{equation*}
where we use the simplified r.v. definition
\begin{equation*}
    X_{d, m, w}: \Omega \rightarrow \mathbb{R},
\end{equation*}
which will be applied in the subsequent analysis. In addition to these features, we define binary response variable $Y: \Omega \rightarrow \{-1, 1\}$ and for each observation $\omega_i$, the realized response is 
\begin{equation*}
    \left\{ y_i = Y(\omega_i) \mid i = 1, \cdots, n\right\}.
\end{equation*}
\textit{W.l.o.g.} we assume that a class label of $-1$ indicates a failure, whereas a class label of $1$ indicates normal operation. 
\begin{remark}
    Each observation is a single time-stamped record containing all standardized weather features across sites, the concurrent power flow information, and the corresponding operational status. We assume the set of records used for estimation is i.i.d.. Observations do not influence one another and all are drawn from the same distribution. This assumption applies across samples, not within a record. Features measured at the same time may be correlated without violating i.i.d.. To support this assumption in practice, we assemble the dataset, for example, by using non-overlapping time windows or event-level sampling. 
\end{remark} 

\subsubsection{Weighted feature selection} \label{subsubsec:weighted_feature_prop}
We begin our discussion by briefly introducing the theorem presented in \cite{biau2012analysis}. Throughout the discussion, we consider the covariate vector as a random variable
\begin{equation*}
    \mathbf{X}_{\rm{all}}: \Omega \rightarrow [0, 1]^{\kappa},\ \kappa \geq 2,
\end{equation*}
which is uniformly distributed on the $\kappa$-dimensional unit cube, and a square integrable response $Y \in \mathbb{R}$. Define a target function $\rm{tf}: [0 ,1]^{\kappa} \rightarrow \mathbb{R}$ as
\begin{equation*}
    \rm{tf}(\mathbf{x_{\rm{all}}}) = \mathbb{E}[Y | \mathbf{X}_{\rm{all}} = \mathbf{x}_{\rm{all}}], \quad \mathbf{x}_{\rm{all}} \in [0, 1]^{\kappa}.
\end{equation*}
It assigns to every covariate vector $\mathbf{x}_{\rm{all}}$ the ideal prediction, the mean value of $Y$ we would expect if the explanatory variables were fixed at $\mathbf{x}_{\rm{all}}$. Practically, we never observe $\rm{tf}$ itself. Instead, we approximate it using a statistical model constructed from data. In this section, we assume that each individual model is a single decision tree trained to approximate the unknown target function. Given a training sample $\mathcal{DS}_n$ of size $n$ and a randomization variable $\Lambda$ that specifies how a tree is grown, denote by $\hat{\rm{t}}\rm{f}_n(\mathbf{x}_{\rm{all}}, \Lambda, \mathcal{DS}_n)$, the prediction of a single tree at the covariate value $\mathbf{x}_{\rm{all}}$. The RF estimator is the conditional expectation
\begin{equation*}
    \overline{\rm{tf}}_n(\mathbf{X}_{\rm{all}}) = \mathbb{E}_{\Lambda}[\hat{\rm{t}}\rm{f}(\mathbf{X}_{\rm{all}}, \Lambda, \mathcal{DS}_n) | \mathcal{DS}_n],
\end{equation*}
interpreted as the average of infinitely many independent trees built from the same data set but with independent realizations of $\Lambda$. Let $\Sigma \subseteq \{1, \cdots, \kappa\}$ be the strong features index set, with cardinality $\varsigma > 1$. By definition, once the features in $\Sigma$ are fixed, the remaining $\kappa - \varsigma$ features carry no information about $Y$. Formally,
\begin{equation*}
    \rm{tf}(\mathbf{x_{\rm{all}}}) = \rm{tf}^{\star}(\mathbf{x}_{\Sigma}); \, \text{a.s}, \, \text{where } \, \rm{tf}^{\star}: [0, 1]^{\varsigma} \rightarrow \mathbb{R}, 
\end{equation*}
and $\rm{tf}^*(\mathbf{x}_{\Sigma})$ is the function $\rm{tf}$ restricted to the subspace of strong features only. During the single tree construction, the randomized tree growing procedure chooses, at every split node, a feature $j \in \{1, \cdots, \kappa\}$ with probability $\pi_{n, j}^{\Sigma}$. If the strong set $\Sigma$ is known, these probabilities satisfy
\begin{equation}
    \pi_{n, j}^{\Sigma} = \begin{cases}
        1/\varsigma , & \text{if } j \in \Sigma \\
        0, & \text{otherwise}.
    \end{cases} \label{eq:coor_perfect_samp_dist}
\end{equation}
In practice, the exact strong feature set $\Sigma$ is generally unknown. 
The randomized tree growing procedure approximates the ideal selection using the training data, resulting in approximate selection probabilities \cite{biau2012analysis}:
\begin{equation}
    \pi_{n, j}^{\Sigma} = \begin{cases}
        \frac{1}{\varsigma} \left( 1 + \xi_{n, j} \right), & \text{if } j \in \Sigma \\
        \xi_{n, j}, & \text{otherwise}.
    \end{cases} \label{eq:coor_gen_samp_dist}
\end{equation}
Throughout the following discussion we assume that the data-driven weights are consistent, in the sense that $\xi_{n,j} \rightarrow 0$ as $n \rightarrow \infty$. The quantity $\xi_{n,j}$ measures how far the actual probability $\pi_{n, j}^{\Sigma}$ is from the ideal value $1/\varsigma$. Intuitively, a larger training sample provides stronger evidence about which features are truly informative. During the construction of the decision tree, each split feature is selected independently according to the probability distribution $\{ \pi_{n, j}^{\Sigma} \}_{j = 1}^{\kappa}$, regardless of the current depth of the tree or previously selected split directions.
\begin{theorem} \label{thm:rf_mse_bound}
    \cite{biau2012analysis} $\mathbf{X}_{\rm{all}}$ is uniformly distributed on $[0, 1]^{\kappa}$, $\rm{tf}^{\star}(\mathbf{x}_{\varsigma})$ is L-Lipschitz on $[0, 1]^{\varsigma}$ and, for all $\bm{x}_{\rm{all}} \in [0, 1]^{\kappa}$, the variance of $Y$ given $\mathbf{X}_{\rm{all}}$ is bounded by some positive constant $\sigma^2$ such that $\rm{var}[Y|\mathbf{X}_{\rm{all}} = \mathbf{x}_{\rm{all}}] \leq \sigma^2, \forall \mathbf{x}_{\rm{all}} \in [0, 1]^{\kappa}$. Then, if $\pi_{n, j}^{\Sigma} = (1/\varsigma)(1 + \xi_{n, j})$ for $j \in \Sigma$, letting $\varpi_{n} = \min_{j \in \varsigma}\xi_{n, j}$, we have the MSE
    \begin{equation} \label{eq:mse_total}
        \mathbb{E}[\overline{\rm{tf}}_n(\mathbf{X}_{\rm{all}}) - \rm{tf}(\mathbf{X}_{\rm{all}})]^2 \leq \Xi_n \frac{|\mathcal{I}_{\rm{ter}, n}|}{n} + \frac{2 \varsigma L^2}{|\mathcal{I}_{\rm{ter,n}}|^{\frac{0.75}{\varsigma \log 2}(1 + \varpi_n)}},
    \end{equation}
    where
    \begin{align} \label{eq:mse_variance_total}
        \Xi_n & = \frac{288}{\pi}\left(\frac{\pi \log 2}{16}\right)^{|\varsigma|/2\kappa} \! \! \! \! \sigma^2\left(\frac{|\varsigma|^2}{|\varsigma| - 1}\right)^{|\varsigma|/2\kappa} \! \! \! \!(1 + \xi_n) \notag \\
        & + 2e^{-1}\left[\sup_{\bm{x}_{\rm{all}}\in[0, 1]^{\kappa}} \rm{tf}^2(\mathbf{x}_{\rm{all}})\right]
    \end{align}
    and
    \begin{equation} \label{eq:mse_variance_component}
        \xi_{n} = \prod_{j \in \varsigma}\left[ (1 + \xi_{n,j})^{-1}\left( 1 - \frac{\xi_{n,j}}{\varsigma - 1}\right)^{-1} \right]^{1/2\kappa} - 1.
    \end{equation}
    Here $\mathcal{I}_{\rm{ter,n}}$ denotes the leaf nodes and $L$ is the L-Lipschitz constant. The sequence $(\xi_{n})$ depends on the sequences $\{ (\xi_{n, j}): j \in \varsigma \}$ only and tends to $0$ as $n$ tends to infinity. 
\end{theorem}
We now specify the sampling schemes used in our analysis.
\begin{definition}[Uniform sampling]\label{def:uniform_samp}
    Under \textit{uniform sampling} the feature selection probabilities satisfy
    \begin{equation*}
        \pi_{n, j}^{\Sigma} = 1/\kappa, \forall j \in \{1, \cdots, \kappa\}.
    \end{equation*}
\end{definition}
\begin{definition}[$\bm{\delta}$-targeted sampling]\label{def:targeted_samp}
    Under $\delta$-\textit{targeted sampling} the feature selection probabilities satisfy
    \begin{equation*}
        \pi_{n, j}^{\Sigma} = \begin{cases}
            \delta_j/\kappa, & \text{if } j \in \Sigma \\
            \left( 1 - \sum_{j \in \Sigma} \frac{\delta_{n, j}}{\kappa}\right)/{\left( \kappa - \varsigma \right)}, & \text{otherwise},
        \end{cases}
    \end{equation*}
    where $\sum_{j \in \Sigma} \delta_j \leq \kappa$; $\delta_j \geq 1$, $\forall j \in \Sigma$ and $\exists j \in \Sigma$ such that $\delta_j > 1$.
\end{definition}
\begin{proposition}\label{prop:small_mse_targeted_samp}
    For any sequence of leaf size $|\mathcal{I}_{\rm{ter,n}}|$ satisfying $|\mathcal{I}_{\rm{ter,n}}| \rightarrow \infty$ and $|\mathcal{I}_{\rm{ter,n}}|/n \rightarrow 0$, $\delta$-targeted sampling with factor $1 < \delta_{n,j} < \kappa/2$, $\forall n$ on features $j$, $j \in \Sigma$ yields a strictly smaller non-asymptotic upper bound on MSE than uniform sampling. 
\end{proposition}
\begin{proof}
    See Appendix \ref{appendix:benefits_weighted_feature_selection}.
\end{proof}
Proposition~\ref{prop:small_mse_targeted_samp} demonstrates that incorporating appropriately guided prior information into the feature selection process results in a strictly lower bound on the MSE compared to random feature selection. 

\subsubsection{Weighted MS decision tree ensemble (WMSDTE)} \label{subsubsec:ensemble_model_wmsdt}

Under the standard RF feature sub-sampling with uniform sampling, each split draws a candidate subset of features of specified size uniformly at random from those available at the node, and the split applied is the one that minimizes impurity within that subset \cite{biau2012analysis, scikitlearn2011python}. This offers the chance to select strong features when they are sampled, but it does not ensure that strong features are considered at every split, especially when the total number of features is large relative to the candidate set size \cite{biau2012analysis}. The $\delta$-targeted sampling increases the probability that at least one strong features appears in each candidate set. In practice, the strong features set is unknown, so we estimate relative importance from data and prior knowledge. We approach this through a two level weighting scheme that exploits the multi-source setting: Site level weights across measurement sources and, within each source, feature weights across variables. These weights define the feature sampling distribution used at each split and provide a feature-level data fusion framework. 

\paragraph{Probabilistic feature subset splitting in binary decision trees} We define the r.v. $\bm{X}: \Omega \rightarrow \mathcal{X}$ by collecting all single-feature r.v.s $X_{d, m, w}$. Formally, for each observation $\omega \in \Omega$, 
\begin{equation*}
    \bm{X}(\omega) = \big( X_{d, m, w}(\omega) \big)_{d \in \mathcal{D}, m \in \mathcal{M}, w \in \mathcal{W}_m},
\end{equation*}
where $\mathcal{X} \subseteq \mathbb{R}^{D \cdot W}$. Thus, each realization $\bm{X}(\omega)$ is an element of $\mathcal{X}$ containing all features from $D$ measurement sites and $M$ categories. Then the realized dataset can also be written as $\left\{ \bm{x}_{i} = \bm{X}(\omega_i) \mid i = 1, \cdots, n \right\}$. Consider the binary classification setting where the class labels are $y_i \in \{-1, 1\}$. We have $n$ i.i.d. samples $\left\{ \left( \bm{x}_i, y_i \right) \right\}_{i = 1}^n$, drawn from an unknown distribution $\Pi$ on $\mathcal{X} \times \{-1, 1\}$, where the combined feature set from all measurement sites $\bm{x}_{i} \in \mathcal{X} \subseteq \mathbb{R}^{D \cdot W}$. We assign a weight distribution $\left\{ \pi_{\omega_i}^{\Omega} \mid i = 1, \cdots, n \right\}$ across the realized dataset, where $\sum_{i = 1}^{n} \pi_{\omega_i}^{\Omega} = 1$. Let $\nu$ denote a node in the decision tree, and let $\mathcal{I}_{\nu} \subseteq \{ 1, \cdots, n \}$ be the set of sample indices that reach node $\nu$. For node $\nu$, let
\begin{equation}
    \text{pr}(\nu) = \frac{\sum_{i \in \mathcal{I}_{\nu}}\pi_{\omega_i}^{\Omega} \bm{1}\left(y_i = 1\right)}{\sum_{i \in \mathcal{I}_{\nu}}\pi_{\omega_i}^{\Omega}} \label{eq:weighted_pr_nu},
\end{equation}
where $\bm{1}\{\diamond\}$ is the indicator function and $\text{pr}(\nu)$ is the weighted proportion of samples in node $\nu$ that belong to class $1$ \cite{breiman_cart_2}. The Gini impurity of node $\nu$ is defined as $GI_{\nu} = 2 \left[\text{pr}(\nu)\right]\left[1 - \text{pr}(\nu)\right]$ \cite{breiman_cart_2}. Unless otherwise indicated, we assume a uniform weight distribution over the samples, i.e., $\pi_{\omega_i}^{\Omega} = 1/n, \forall i$. Then
\begin{equation}
    \text{pr}(\nu) = \frac{|\{i \in \mathcal{I}_{\nu}: y_i = 1\}|}{|\mathcal{I}_{\nu}|} \label{eq:unweighted_pr_nu}
\end{equation}
where $|\diamond|$ denotes the cardinality of the set $\diamond$. For each measurement site $d \in \mathcal{D}$, we assign a strictly positive sampling weight 
\begin{equation*}
    \pi_{d}^D > 0 \quad s.t. \sum_{d \in \mathcal{D}} \pi_d^{D} = 1 \quad \forall d \in \mathcal{D}.
\end{equation*}
Then $\left\{ \pi_{d}^{\mathcal{D}} \right\}_{d \in \mathcal{D}}$ forms a probability distribution over the measurement sites. Similarly $\left\{ \pi_{m}^{\mathcal{M}}\right\}_{m \in \mathcal{M}}$ forms a strictly positive probability distribution over the feature categories. To generate a candidate subset for the best possible split at node $\nu$, we randomly sample a subset of measurement sites $\mathcal{D}_{\nu} \subseteq \mathcal{D}$ without replacement, according to the probability distribution $\left\{ \pi_{d}^{\mathcal{D}} \right\}_{d \in \mathcal{D}}$. Similarly, we sample a subset of feature categories $\mathcal{M}_{\nu} \subseteq \mathcal{M}$ without replacement, according to the probability distribution $\left\{ \pi_{m}^{\mathcal{M}} \right\}_{m \in \mathcal{M}}$. The candidate subset at the node $\nu$ is defined as:
\begin{equation*}
    \mathcal{X}_{\nu} = \bigcup_{d \in \mathcal{D}_{\nu}} \bigcup_{m \in \mathcal{\mathcal{M}_{\nu}}} \mathcal{W}_{m}^{d},
\end{equation*}
where $\mathcal{W}_{m}^{d}$ denotes the set of features belonging to category $m$ at measurement site $d$. The corresponding sub-vector of $\bm{X}(\omega)$, restricted to the indices in $\mathcal{X}_{\nu}$ is given as:
\begin{equation*}
    \bm{X}_{\nu}(\omega) = \big( X_{d, m, w}(\omega) \big)_{d \in \mathcal{D}_{\nu}, m \in \mathcal{M}_{\nu}, w \in \mathcal{W}_m^{d}}.
\end{equation*}
At each node $\nu$, $\varkappa \in \mathcal{X}_{\nu}$ and threshold $\tau \in \mathbb{R}$, define the left (right) children of node $\nu$ as $\nu^{L}$ ($\nu^{R}$) and the set of sample indices that reach node $\nu^{L}$ ($\nu^{R}$) as 
\begin{equation*}
    \mathcal{I}_{\nu}^{L}(\varkappa, \tau) = \left\{ i \in \mathcal{I}_{\nu}: x_{i, \varkappa} < \tau \right\} \halfquad \left( \mathcal{I}_{\nu}^{R}(\varkappa, \tau) = \mathcal{I}_{\nu} \backslash \mathcal{I}_{\nu}^{L}(\varkappa, \tau) \right).
\end{equation*}
We select
\begin{equation}
    \left( \varkappa^*, \tau^* \right) = \arg \min_{\varkappa \in \mathcal{X}_{\nu}, \tau \in \mathbb{R}} Q\left( \mathcal{I}_{\nu}; \varkappa, \tau \right), \label{eq:node_split_decision_criteria}
\end{equation}
where $Q\left( \mathcal{I}_{\nu}; \varkappa, \tau \right)$ is the post-split impurity as \cite{breiman_cart_2}
\begin{equation*}
    \frac{| \mathcal{I}_{\nu}^{L}(\varkappa, \tau) |}{| \mathcal{I}_{\nu} |} GI_{\nu^{L}} + \frac{| \mathcal{I}_{\nu}^{R}(\varkappa, \tau) |}{| \mathcal{I}_{\nu} |} GI_{\nu^{R}}.
\end{equation*}

\paragraph{Sampling weights for measurement sites}
In practical circumstances, it is reasonable to assume that weather data collected from sites closer to the location of interest are more indicative of location conditions than data from distant sites. To capture this effect, we assign larger weights to measurement sites nearer the location of interest and smaller weights to those farther away. We set each site's weight to be inversely proportional to its distance from the location of interest and normalize these weights so that they form a probability distribution over the sites. Specifically, let $\left( \phi_{d}, \lambda_{d} \right)$ denote the coordinates of measurement site $d$, where $\phi_{d}$ is the latitude and $\lambda_d$ is the longitude and $\left( \phi_{d}, \lambda_{d} \right)$ represents the location of the substation. Let $r_d$ be the Euclidean distance from site $d$ to the substation:
\begin{equation*}
    r_d = \sqrt{\left( \phi_d - \phi_0 \right)^2 + \left( \lambda_d - \lambda_0 \right)^2}
\end{equation*}
where latitudes and longitudes are treated as coordinates in a plane. We then assign each site $d$ a strictly positive weight $\pi_d^{\mathcal{D}}$ inversely proportional to $r_d$:
\begin{equation}
    \pi_d^{\mathcal{D}} = \frac{1/r_d}{\sum_{k \in \mathcal{D}}\left( 1/r_k \right)}. \label{eq:site_wise_sample_weight}
\end{equation}

\paragraph{Sampling weights for feature categories} To guide the allocation of category weights, we leverage logistic regression coefficients as a measure of each category's predictive strength; i.e., the magnitude of a coefficient reflects how strongly the category $m \in \mathcal{M}$ correlates with the response variables \cite{kaufman1996comparing}. This interpretation is clearest when cross category correlations are weak \cite{gelman_daurmhm_3, gelman_daurmhm_5}. In practice, multicollinearity remains: We therefore treat these coefficients as heuristic weights used to set a relative ordering of category importance, rather than as precise effect sizes. For each observation $i$ and category $m$, we create a single summary statistic by averaging all features associated with that category across all measurement sites. Specifically,
\begin{equation*}
    z'_{i, m} = \frac{1}{|\mathcal{D}|\cdot |\mathcal{W}_m|}\sum_{d \in \mathcal{D}} \sum_{w \in \mathcal{W}_m} x_{i, d, m, w},
\end{equation*}
for $m \in \mathcal{M},\ i = 1, \cdots, n$ yielding a single predictor $z'_{i, m}$ for each category $m$ in every sample $i$. Using $\left\{ z'_{i, m} \right\}_{m = 1}^{M}$ as predictors, we fit a logistic regression model to predict the binary response $y_i \in \{ -1, 1 \}$:
\begin{equation*}
    \log \left( \frac{\text{pr}_i}{1 - \text{pr}_i} \right) = \alpha_0 + \sum_{m = 1}^{M}\alpha_m z'_{i, m}, 
\end{equation*}
where
\begin{equation*}
    \text{pr}_i = \text{Pr}\left( Y_i = 1 | z'_{i, 1}, \cdots, z'_{i, M} \right).
\end{equation*}
Suppose the fitted coefficients are $\hat{\alpha}_0, \hat{\alpha}_1, \cdots, \hat{\alpha}_M$. The magnitude $|\hat{\alpha}_m|$ for $m = 1, \cdots, M$ reflects the strength of this relationship. We then assign each feature category $m$ a strictly positive weight $\pi_{m}^{\mathcal{M}}$:
\begin{equation*}
    \pi_{m}^{\mathcal{M}} = \frac{|\hat{\alpha}_m|}{\sum_{h = 1}^M |\hat{\alpha}_h|} \quad \text{for } m = 1, \cdots, M.
\end{equation*}

\paragraph{Construction of WMSDTE}
We grow each decision tree by recursively applying the node-splitting procedure described above \cite{breiman_cart_2}. Analogous to the RF approach, we bootstrap $n_h$ samples from the training sample set $\left\{ \left(\bm{x}_i, y_i \right) \mid i = 1, \cdots, n_{\text{train}} \right\}$ for each tree at the start of its construction. This process continues until a predefined maximum depth $H_{\max}$ is reached. To form the final ensemble, we independently construct $J$ such weighted MS decision trees $\left\{ h_1, \cdots, h_{J} \right\}$. For a new feature vector $\bm{x}_{\text{new}}$, each tree $h_t(\bm{x_{\text{new}}})$ makes a binary prediction in $\{-1, 1\}$. We aggregate these predictions by majority vote \cite{breiman_cart_2}:
\begin{equation*}
    \hat{y}(\bm{x}_{\text{new}}) = \arg \max_{y \in \{0, 1\}} \sum_{j = 1}^{J}\bm{1}\left\{ h_j(\bm{x_{\text{new}}}) = y \right\}.
\end{equation*}
The component failure probability can be estimated by 
\begin{equation*}
    \hat{\text{P}}\text{r}(\bm{x}_{\text{new}}) = \frac{1}{J}\sum_{j = 1}^{J} \hat{\text{P}}\text{r}_j(\bm{x}_{\text{new}}),
\end{equation*}
where $\hat{\text{P}}\text{r}_j(\bm{x}_{\text{new}})$ is the leaf-level probability of failure returned by the $j^{th}$ tree. The leaf-level probability is the fraction of training samples in leaf ``$\text{leaf}_l$'' (the same leaf in which $\bm{x}_{\text{new}}$ is located) that belong to the failure class, and is calculated by
\begin{equation*}
    \frac{\sum_{i \in \text{leaf}_l } \bm{1}\left\{ y _i = -1\right\}}{|\left\{ \bm{x}_i: i \in \text{leaf}_l  \right\}|}.
\end{equation*}
WMSDTE is illustrated in Figure \ref{fig:wmsdt} and formalized in Algorithm \ref{alg:wmsdt}.

\begin{figure}
    \centering
    \resizebox{\linewidth}{!}{\begin{tikzpicture}[font=\small, node distance=1cm and 0.5cm, >=Stealth]

\node[draw, rounded corners, fill=blue!20, minimum width=1cm, minimum height=0.5cm, text centered] at (-1,1) (e1ds1) {$\mathcal{D}_{a}$};
\node[draw, rounded corners, fill=blue!30, minimum width=1cm, minimum height=0.5cm, text centered] at (-1, 0.5) (e1ds2) {$\mathcal{D}_{b}$};

\node[draw, rounded corners, fill=blue!30, minimum width=1cm, minimum height=0.5cm, text centered] at (-1, -0.5) (e2ds2) {$\mathcal{D}_{b}$};
\node[draw, rounded corners, fill=blue!30, minimum width=1cm, minimum height=0.5cm, text centered] at (-1, -1) (e2ds3) {$\mathcal{D}_{c}$};

\node[draw, rounded corners, fill=blue!30, minimum width=1cm, minimum height=0.5cm, text centered] at (-1, -2) (e3ds2) {$\mathcal{D}_{b}$};
\node[draw, rounded corners, fill=blue!20, minimum width=1cm, minimum height=0.5cm, text centered] at (-1, -2.5) (e3ds4) {$\mathcal{D}_{d}$};

\node[draw, rounded corners, fill=blue!20, minimum width=1cm, minimum height=0.5cm, text centered] at (-1, -3.5) (e4ds1) {$\mathcal{D}_{a}$};
\node[draw, rounded corners, fill=blue!30, minimum width=1cm, minimum height=0.5cm, text centered] at (-1, -4) (e3ds3) {$\mathcal{D}_{c}$};

\node[draw, fill=cyan!50, minimum width=1cm, minimum height=0.5cm, text centered] at (1, 1) (e1gfh) {$\mathcal{M}_{a}$};
\node[draw, fill=cyan!25, minimum width=1cm, minimum height=0.5cm, text centered] at (1, 0.5) (e1gfhf) {...};
\node[draw, fill=cyan!70, minimum width=1cm, minimum height=0.5cm, text centered] at (2, 1) (e1gfw) {$\mathcal{M}_{b}$};
\node[draw, fill=cyan!35, minimum width=1cm, minimum height=0.5cm, text centered] at (2, 0.5) (e1gfwf) {...};

\node[draw, fill=cyan!70, minimum width=1cm, minimum height=0.5cm, text centered] at (1, -0.5) (e2gfw) {$\mathcal{M}_{b}$};
\node[draw, fill=cyan!35, minimum width=1cm, minimum height=0.5cm, text centered] at (1, -1) (e2gfwf) {...};
\node[draw, fill=cyan!30, minimum width=1cm, minimum height=0.5cm, text centered] at (2, -0.5) (e2gfp) {$\mathcal{M}_{c}$};
\node[draw, fill=cyan!15, minimum width=1cm, minimum height=0.5cm, text centered] at (2, -1) (e2gfpf) {...};

\node[draw, fill=cyan!50, minimum width=1cm, minimum height=0.5cm, text centered] at (1, -2) (e3gfh) {$\mathcal{M}_{a}$};
\node[draw, fill=cyan!25, minimum width=1cm, minimum height=0.5cm, text centered] at (1, -2.5) (e3gfhf) {...};
\node[draw, fill=cyan!70, minimum width=1cm, minimum height=0.5cm, text centered] at (2, -2) (e3gfw) {$\mathcal{M}_{b}$};
\node[draw, fill=cyan!35, minimum width=1cm, minimum height=0.5cm, text centered] at (2, -2.5) (e3gfwf) {...};

\node[draw, fill=cyan!50, minimum width=1cm, minimum height=0.5cm, text centered] at (1, -3.5) (e4gfh) {$\mathcal{M}_{a}$};
\node[draw, fill=cyan!25, minimum width=1cm, minimum height=0.5cm, text centered] at (1, -4) (e4gfhf) {...};
\node[draw, fill=cyan!10, minimum width=1cm, minimum height=0.5cm, text centered] at (2, -3.5) (e4gfo) {$\mathcal{M}_{d}$};
\node[draw, fill=cyan!5, minimum width=1cm, minimum height=0.5cm, text centered] at (2, -4) (e4gfof) {...};

\draw[->, line width = 1.2pt, draw = purple] (-0.5,0.75) -- (0.5,0.75);
\draw[->, line width = 1.2pt, draw = purple] (-0.5,-0.75) -- (0.5,-0.75);
\draw[->, line width = 1.2pt, draw = purple] (-0.5,-2.25) -- (0.5,-2.25);
\draw[->, line width = 1.2pt, draw = purple] (-0.5,-3.75) -- (0.5,-3.75);

\draw[->, densely dotted, line width = 1.2pt, draw = purple] (2.5,0.75) -- (3.5,-0.75);
\draw[->, densely dotted, line width = 1.2pt, draw = purple] (2.5,-0.75) -- (3.5,-1.25);
\draw[->, densely dotted, line width = 1.2pt, draw = purple] (2.5,-2.25) -- (3.5,-1.75);
\draw[->, densely dotted, line width = 1.2pt, draw = purple] (2.5,-3.75) -- (3.5,-2.25);

\draw[-, dotted, line width = 1.2pt, draw = purple] (3.5,1.25) -- (8.75,1.25);
\draw[-, dotted, line width = 1.2pt, draw = purple] (3.5,1.25) -- (3.5,-4.25);
\draw[-, dotted, line width = 1.2pt, draw = purple] (8.75,1.25) -- (8.75,-4.25);
\draw[-, dotted, line width = 1.2pt, draw = purple] (3.5,-4.25) -- (8.75,-4.25);

\begin{scope}[on background layer]
  \draw[-, thick] (4.5,0.75) -- (4,0); 
  \draw[-, thick] (4.5,0.75) -- (5,0);
  \draw[-, thick] (5,0) -- (4.5,-0.75);
  \draw[-, thick] (5,0) -- (5.5,-0.75);
  \draw[-, dotted, thick] (4.0,0) -- (4.0,-0.5);
  \draw[-, dotted, thick] (4.5,-0.75) -- (4.5,-1.25);
  \draw[-, dotted, thick] (5.5,-0.75) -- (5.5,-1.25);
\end{scope}

\begin{scope}
    \clip (4.5,0.75) circle(0.25cm);
    \fill[blue!30] (4.25,0.75) rectangle (4.75,1);   
    \fill[cyan!70] (4.25,0.75) rectangle (4.75,0.5); 
\end{scope}
\draw (4.5,0.75) circle(0.25cm); 

\begin{scope}
    \clip (4,0) circle(0.25cm);
    \fill[blue!20] (3.75,0) rectangle (4.25,0.25);    
    \fill[cyan!50] (3.75,0) rectangle (4.25,-0.25);  
\end{scope}
\draw (4,0) circle(0.25cm); 

\begin{scope}
    \clip (5,0) circle(0.25cm);
    \fill[blue!30] (4.75,0) rectangle (5.25,0.25);    
    \fill[cyan!10] (4.75,0) rectangle (5.25,-0.25);  
\end{scope}
\draw (5,0) circle(0.25cm); 

\begin{scope}
    \clip (4.5,-0.75) circle(0.25cm);
    \fill[blue!20] (4.25,-0.75) rectangle (4.75,-0.5);   
    \fill[cyan!70] (4.25,-0.75) rectangle (4.75,-1); 
\end{scope}
\draw (4.5,-0.75) circle(0.25cm); 

\begin{scope}
    \clip (5.5,-0.75) circle(0.25cm);
    \fill[blue!30] (5.25,-0.75) rectangle (5.75,-0.5);    
    \fill[cyan!70] (5.25,-0.75) rectangle (5.75,-1);  
\end{scope}
\draw (5.5,-0.75) circle(0.25cm); 

\begin{scope}[on background layer]
  \draw[-, thick] (7,0.75) -- (6.5,0); 
  \draw[-, thick] (7,0.75) -- (7.5,0);
  \draw[-, thick] (7.5,0) -- (7,-0.75);
  \draw[-, thick] (7.5,0) -- (8,-0.75);
  \draw[-, dotted, thick] (6.5,0) -- (6.5,-0.5);
  \draw[-, dotted, thick] (7,-0.75) -- (7,-1.25);
  \draw[-, dotted, thick] (8,-0.75) -- (8,-1.25);
  \draw[-, dotted, thick] (5.5,0) -- (6,0);
\end{scope}

\begin{scope}
    \clip (7,0.75) circle(0.25cm);
    \fill[blue!20] (6.75,0.75) rectangle (7.25,1);   
    \fill[cyan!50] (6.75,0.75) rectangle (7.25,0.5); 
\end{scope}
\draw (7,0.75) circle(0.25cm); 

\begin{scope}
    \clip (6.5,0) circle(0.25cm);
    \fill[blue!30] (6.25,0) rectangle (6.75,0.25);    
    \fill[cyan!70] (6.25,0) rectangle (6.75,-0.25);  
\end{scope}
\draw (6.5,0) circle(0.25cm); 

\begin{scope}
    \clip (7.5,0) circle(0.25cm);
    \fill[blue!30] (7.25,0) rectangle (7.75,0.25);    
    \fill[cyan!50] (7.25,0) rectangle (7.75,-0.25);  
\end{scope}
\draw (7.5,0) circle(0.25cm); 

\begin{scope}
    \clip (7,-0.75) circle(0.25cm);
    \fill[blue!30] (6.75,-0.75) rectangle (7.25,-0.5);   
    \fill[cyan!70] (6.75,-0.75) rectangle (7.25,-1); 
\end{scope}
\draw (7,-0.75) circle(0.25cm); 

\begin{scope}
    \clip (8,-0.75) circle(0.25cm);
    \fill[blue!20] (7.75,-0.75) rectangle (8.25,-0.5);    
    \fill[cyan!50] (7.75,-0.75) rectangle (8.25,-1);  
\end{scope}
\draw (8,-0.75) circle(0.25cm); 

\begin{scope}[on background layer]
  \draw[-, thick] (4.5,-1.75) -- (4,-2.5); 
  \draw[-, thick] (4.5,-1.75) -- (5,-2.5);
  \draw[-, thick] (5,-2.5) -- (4.5,-3.25);
  \draw[-, thick] (5,-2.5) -- (5.5,-3.25);
  \draw[-, dotted, thick] (4.0,-2.5) -- (4.0,-3);
  \draw[-, dotted, thick] (4.5,-3.25) -- (4.5,-3.75);
  \draw[-, dotted, thick] (5.5,-3.25) -- (5.5,-3.75);
\end{scope}

\begin{scope}
    \clip (4.5,-1.75) circle(0.25cm);
    \fill[blue!30] (4.25,-1.75) rectangle (4.75,-1.5);   
    \fill[cyan!50] (4.25,-1.75) rectangle (4.75,-2); 
\end{scope}
\draw (4.5,-1.75) circle(0.25cm); 

\begin{scope}
    \clip (4,-2.5) circle(0.25cm);
    \fill[blue!30] (3.75,-2.5) rectangle (4.25,-2.25);    
    \fill[cyan!70] (3.75,-2.5) rectangle (4.25,-2.75);  
\end{scope}
\draw (4,-2.5) circle(0.25cm); 

\begin{scope}
    \clip (5,-2.5) circle(0.25cm);
    \fill[blue!20] (4.75,-2.5) rectangle (5.25,-2.25);    
    \fill[cyan!50] (4.75,-2.5) rectangle (5.25,-2.75);  
\end{scope}
\draw (5,-2.5) circle(0.25cm); 

\begin{scope}
    \clip (4.5,-3.25) circle(0.25cm);
    \fill[blue!30] (4.25,-3.25) rectangle (4.75,-3);   
    \fill[cyan!50] (4.25,-3.25) rectangle (4.75,-3.5); 
\end{scope}
\draw (4.5,-3.25) circle(0.25cm); 

\begin{scope}
    \clip (5.5,-3.25) circle(0.25cm);
    \fill[blue!30] (5.25,-3.25) rectangle (5.75,-3);    
    \fill[cyan!70] (5.25,-3.25) rectangle (5.75,-3.5);  
\end{scope}
\draw (5.5,-3.25) circle(0.25cm); 

\begin{scope}[on background layer]
  \draw[-, thick] (7,-1.75) -- (6.5,-2.5); 
  \draw[-, thick] (7,-1.75) -- (7.5,-2.5);
  \draw[-, thick] (7.5,-2.5) -- (7,-3.25);
  \draw[-, thick] (7.5,-2.5) -- (8,-3.25);
  \draw[-, dotted, thick] (6.5,-2.5) -- (6.5,-3);
  \draw[-, dotted, thick] (7,-3.25) -- (7,-3.75);
  \draw[-, dotted, thick] (8,-3.25) -- (8,-3.75);
  \draw[-, dotted, thick] (5.5,-2.5) -- (6,-2.5);
\end{scope}

\begin{scope}
    \clip (7,-1.75) circle(0.25cm);
    \fill[blue!20] (6.75,-1.75) rectangle (7.25,-1.5);   
    \fill[cyan!50] (6.75,-1.75) rectangle (7.25,-2); 
\end{scope}
\draw (7,-1.75) circle(0.25cm); 

\begin{scope}
    \clip (6.5,-2.5) circle(0.25cm);
    \fill[blue!30] (6.25,-2.5) rectangle (6.75,-2.25);    
    \fill[cyan!70] (6.25,-2.5) rectangle (6.75,-2.75);  
\end{scope}
\draw (6.5,-2.5) circle(0.25cm); 

\begin{scope}
    \clip (7.5,-2.5) circle(0.25cm);
    \fill[blue!30] (7.25,-2.5) rectangle (7.75,-2.25);    
    \fill[cyan!70] (7.25,-2.5) rectangle (7.75,-2.75);  
\end{scope}
\draw (7.5,-2.5) circle(0.25cm); 

\begin{scope}
    \clip (7,-3.25) circle(0.25cm);
    \fill[blue!30] (6.75,-3.25) rectangle (7.25,-3);   
    \fill[cyan!70] (6.75,-3.25) rectangle (7.25,-3.5); 
\end{scope}
\draw (7,-3.25) circle(0.25cm); 

\begin{scope}
    \clip (8,-3.25) circle(0.25cm);
    \fill[blue!20] (7.75,-3.25) rectangle (8.25,-3);    
    \fill[cyan!10] (7.75,-3.25) rectangle (8.25,-3.5);  
\end{scope}
\draw (8,-3.25) circle(0.25cm);

\end{tikzpicture}}
    \vspace{-1em}    
    \caption{Illustration of WMSDT from data (left), via categories (middle), to decision trees (right). Darker shading denotes higher weight for $\mathcal{D}_i$ or $\mathcal{M}_i$.}
    \label{fig:wmsdt}
\end{figure}
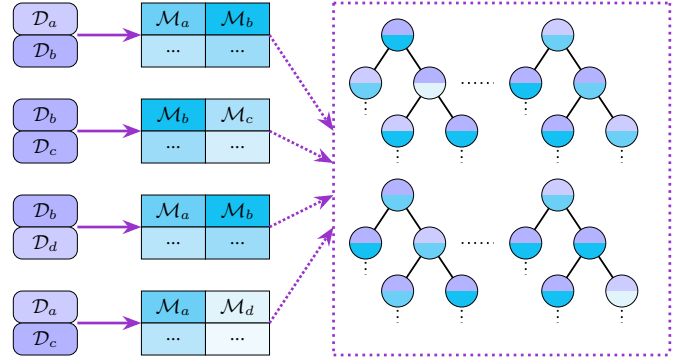

\begin{algorithm}
\caption{WMSDTE}
\label{alg:wmsdt}
\begin{algorithmic}[1]
    \Statex \hspace*{-1.6em} \textbf{Input:} $\left\{ \left(\bm{x}_i, y_i \right) \mid i = 1, \cdots, n_{\text{train}} \right\}$; $\left\{ \pi_d^{\mathcal{D}} \right\}$; $\left\{ \pi_m^{\mathcal{M}} \right\}$; $|\mathcal{D}_{\nu}|$; $|\mathcal{M}_{\nu}|$; $J$; $H_{\max}$; $n_h$;
    \Statex \textit{node $\nu$ splitting criterion}: A node stops splitting if
    \Statex \hspace*{2em} $(1)$ $|\mathcal{X}_{\nu}| < \text{minimum samples split}$ or
    \Statex \hspace*{2em} $(2)$ $\text{tree depth} \geq H_{\max}$
    \Statex \hspace*{-1.6em} \textbf{Output:} $\mathcal{H} = \left\{ h_j \mid j = 1, \cdots, J \right\}$
\State \textbf{Initialize} the set for storing constructed trees: $\mathcal{H} \leftarrow \varnothing$
\For{$j = 1$ to $J$}
    \State Bootstrap sample (size $n_h$) from the training set
    \State Initialize the tree root with $\left\{( \bm{x}_{i}^{j}, y_{i}^j) \mid i = 1, \cdots, n_h \right\}$
    \While{\textit{splitting criterion is not met for node $\nu$}}
        \State Randomly select $|\mathcal{D}_{\nu}|$ number of sites without replacement from $\mathcal{D}$ according to $\left\{ \pi_d^{\mathcal{D}} \right\}$
        \State Randomly select $|\mathcal{M}_{\nu}|$ number of categories without replacement from $\mathcal{M}$ according to $\left\{ \pi_m^{\mathcal{M}} \right\}$
        \State Construct the feature subset at node $\nu$:
            \begin{equation*}
                \mathcal{X}_{\nu} \! = \! \left\{ x_{i, d, m, w} \! \mid \! d \! \in \! \mathcal{D}_{\nu}, m \! \in \! \mathcal{M}_{\nu}, w \! \in \! \mathcal{W}_{m}, i \! = \! 1, \! \cdots \!, n_h \right\}
            \end{equation*}
        \State In $\bm{x}_{\nu}$, find the best split using \eqref{eq:node_split_decision_criteria} and \eqref{eq:unweighted_pr_nu}
        \State Split node $\nu$ into its child nodes
    \EndWhile
    \State Add the resulting decision tree $h_j$ to $\mathcal{H}$
\EndFor
\State {\textbf{return}} $\mathcal{H}$
\end{algorithmic}
\end{algorithm}

\subsubsection{Two cross-stage MS boosting (TCSMSB)} \label{subsubsec:ensemble_model_tsmsb}
In Section \ref{subsubsec:ensemble_model_wmsdt}, we use random feature sampling and site-based weighting to reduce predictor redundancy and balance site influence. We additionally employ bootstrap sampling for prediction variance reduction. To further address the multicollinearity and to speed up the model's computation speed,
in this section, we combine knowledge-based data processing (data-level fusion) with an appropriate dimensionality reduction technique (feature-level fusion), thereby summarizing the data in two distinct ways. We then adopt a boosting framework that adaptively selects between these representations (decision-level fusion).

\paragraph{Category- and site-wise grouping} To handle multiple-site weather datasets and reduce multicollinearity, we apply data-level fusion by organizing related features into groups in two different ways. The first strategy groups features by category. We assume that features in the same category tend to exhibit stronger interdependencies, which we capture and reduce through category-wise grouping. For example, near-surface temperature, net solar and thermal radiation, and sensible heat flux are jointly governed by surface energy exchange, so they co-vary more tightly within the heat group than with the $10$ m u- and v-components of wind in the wind group. For each measurement site $d \in \mathcal{D}$ and each category $m \in \mathcal{M}$, gather the data
\begin{equation*}
    \left\{ \bm{x}_{i, d, m}^{\text{c-w}} = \left(x_{i, d, m, w} \right)_{w \in \mathcal{W}_{m}} \mid i = 1, ..., n\right\}
\end{equation*}
into a single matrix 
\begin{equation*}
    \bar{\bm{x}}_{d, m}^{\text{c-w}} = \left[ ( \bm{x}_{1, d, m}^{\text{c-w}} )^{\top}, \cdots, ( \bm{x}_{n, d, m}^{\text{c-w}} )^{\top} \right]^{\top} \in \mathbb{R}^{n \times |\mathcal{W}_m|}. 
\end{equation*}
We apply PCA to $\bar{\bm{x}}_{d, m}^{\text{c-w}}$ and select the smallest integer $k_{d, m}^{\text{c-w}}$ such that the cumulative variance explained is at least $\gamma^{\text{c-w}}$ where $k_{d, m}^{\text{c-w}} \leq |\mathcal{W}_m|$. We then project $\bar{\bm{x}}_{d, m}^{\text{c-w}}$ onto the top $k_{d, m}^{\text{c-w}}$ principal directions to obtain a dimensionally reduced matrix $\bar{\bm{z}}_{d, m}^{\text{c-w}} \in \mathbb{R}^{n \times k_{d, m}^{\text{c-w}}}$ \cite{jolliffe_pca_2, jolliffe_pca_6}. After repeating for all $m \in \mathcal{M}$, we concatenate the result matrices along their column dimension:
\begin{equation*}
    \bar{\bm{z}}^{\text{c-w}} = \left( \bar{\bm{z}}_{1, 1}^{\text{c-w}}, \cdots, \bar{\bm{z}}_{1, M}^{\text{c-w}}, \cdots, \bar{\bm{z}}_{D, 1}^{\text{c-w}}, \cdots, \bar{\bm{z}}_{D, M}^{\text{c-w}} \right), 
\end{equation*}
and 
\begin{equation*}
    \bar{\bm{z}}^{\text{c-w}} \in \mathbb{R}^{n \times \sum_{d = 1}^D\sum_{m = 1}^M k_{d, m}^{\text{c-w}}}.
\end{equation*}
The second strategy organizes measurements by site rather than by category. We assume that measurements from multiple sites for the same feature are correlated, especially if the sites are close to each other. Thus, by site-wise grouping, we capture and reduce these interdependencies at the site level for each feature. For each feature $w \in \mathcal{W}$, gather the data
\begin{equation*}
    \left\{ \bm{x}_{i, w}^{\text{s-w}} = \left(x_{i, d, w} \right)_{d \in \mathcal{D}} \mid i = 1, ..., n\right\}
\end{equation*}
into a single matrix
\begin{equation*}
    \bar{\bm{x}}_{w}^{\text{s-w}} = \left[ {(\bm{x}_{1, w}^{\text{s-w}})}^{\top}, {(\bm{x}_{2, w}^{\text{s-w}})}^{\top}, \cdots, {(\bm{x}_{n, w}^{\text{s-w}})}^{\top} \right]^{\top} \in \mathbb{R}^{n \times |\mathcal{D}|}. 
\end{equation*}
Following the same dimensional reduction procedure as in the category-wise grouping, we obtain the resulting matrix
\begin{equation*}
    \bar{\bm{z}}^{\text{s-w}} = \left( \bar{\bm{z}}_{1}^{\text{s-w}}, \bar{\bm{z}}_{2}^{\text{s-w}}, \cdots, \bar{\bm{z}}_{W}^{\text{s-w}} \right) \in \mathbb{R}^{n \times \sum_{w = 1}^W k_w^{\text{s-w}}}.
\end{equation*}

\paragraph{Dual-subset feature boosting} Unlike standard boosting algorithms that rely on a single weak learner at each iteration, TCSMSB trains two candidate learners using category and site-wise data groupings \cite{freund1997decision}. The dual perspective acknowledges that, because the weak learners in boosting are often tree-based and intentionally kept shallow, each individual model's capacity to represent complex structures is inherently limited \cite{friedman2000additive, schapire2003boosting}. By alternating between groupings, our method can embed both domain-level and localized signals within these smaller trees. To decide which arrangement best handles the current set of ``harder'' observations, TCSMSB compares the weighted errors of both learners at each iteration and selects the more accurate one to include in the ensemble, while the sample weights are updated according to this choice. As boosting progresses and the weight distribution evolves, the method dynamically exploiting whichever grouping is more effective at reducing misclassification. TCSMSB preserves the benefit of data-adaptive reweighting and allows the algorithm to leverage both grouping structures without committing to either representation in advance. At each iteration of our boosting procedure, we build two WMSDTs as our candidate weak learners: One trained on the category-wise data arrangement and the other on the site-wise arrangement. When building the WMSDT under the category-wise grouping, in each node $\nu$, we sample a subset of measurement sites $\mathcal{D}_{\nu} \subseteq \mathcal{D}$ without replacement according to the distribution $\left\{ \pi_{d}^{\mathcal{D}} \right\}_{d \in \mathcal{D}}$. For each selected sites $d \in \mathcal{D}_{\nu}$, all of the corresponding PCs extracted from the category-wise PCA
\begin{equation*}
    \left( \bar{\bm{z}}_{d, m}^{\text{c-w}} \right)_{d \in \mathcal{D}_{\nu}, m \in \mathcal{M}} \in \mathbb{R}^{n \times \sum_{d \in \mathcal{D}_{\nu}}\sum_{m = 1}^M k_{d, m}^{\text{c-w}}}
\end{equation*}
are included as candidate splitting features at the node $\nu$. Similarly, in the site-wise grouping, we draw a random subset of categories without replacement $\mathcal{M}_{\nu}$ using the distribution $\left\{ \pi_{m}^{\mathcal{M}} \right\}_{m \in \mathcal{M}}$ for node $\nu$ in WMSDT. For each selected category $m \in \mathcal{M}_{\nu}$, we include all site-wise PCs associated with that category
\begin{equation*}
    \left( \bar{\bm{z}}_{w}^{\text{s-w}} \right)_{w \in \mathcal{W}_{m}} \in \mathbb{R}^{n \times \sum_{w \in \mathcal{W}_{m}}k_w^{\text{s-w}}}
\end{equation*}
as candidate features. After training both WMSDTs at each iteration, we select the one that achieves higher weighted accuracy and discard the inferior candidate. The selected weak learner is then assigned a weight according to the standard Real AdaBoost formula, and the observation weights are updated accordingly to emphasize misclassified samples \cite{freund1997decision, friedman2000additive}. These steps iterate until the ensemble is complete and the final classifier is formed as a weighted majority vote of the chosen weak learners, as detailed in Algorithm \ref{alg:tcsmsb}. The component failure probability can be estimated by 
\begin{equation*}
    \hat{\text{P}}\text{r}\left( \bm{x}_{\text{new}} \right) = 1 - \frac{1}{1 + \exp{\left( -2\sum_{j = 1}^{J}\psi_j\left(\bm{x}_{\text{new}}\right) \right)}}
\end{equation*}
where $\psi_{*}(\bm{x})$ is detailed in \eqref{eq:psi}, Algorithm \ref{alg:tcsmsb}.

\begin{algorithm}
\caption{TCSMSB}
\label{alg:tcsmsb}
\begin{algorithmic}[1]
\Statex \hspace*{-1.6em} \textbf{Input:} $\left\{ \left(\bm{x}_i, y_i \right) \mid i = 1, \cdots, n_{\text{train}} \right\}$; $\left\{ \pi_d^{\mathcal{D}} \right\}$; $\left\{ \pi_m^{\mathcal{M}} \right\}$; $\gamma^{\text{c-w}}$; $\gamma^{\text{s-w}}$; $J$; $H_{\max}$;
\Statex \textit{node $\nu$ splitting criterion}: A node stops splitting if
\Statex \hspace*{2em} $(1)$ $|\mathcal{X}_{\nu}| < \text{minimum samples split}$ or
\Statex \hspace*{2em} $(2)$ $\text{tree depth} \geq H_{\max}$
\Statex \hspace*{-1.6em} \textbf{Output:} trained weak learners: $\mathcal{H} = \left\{ h_j \mid j = 1, \cdots, J \right\}$
\Statex \hspace*{2.2em} confidence score: $\Psi = \left\{ \psi_j \mid j = 1, \cdots, J \right\}$
\State Initialize the distribution over the training samples: 
\Statex \hspace*{3.5em} $\left\{ \pi_{\omega}^{\Omega} \right\} = \left\{ \pi_{\omega_i}^{\Omega} = 1/{n_{\text{train}}} \mid i = 1, \cdots, n_{\text{train}} \right\}$
\For{$j = 1$ to $J$}
    \State Initialize the root of $h_{j}^{\text{c-w}}$ with $\left\{ \left(\bm{x}_i, y_i \right) \right\}$ and $\left\{ \pi_{\omega}^{\Omega} \right\}$
    \While{\textit{splitting criterion is not met for node $\nu$}}
        \State Randomly select $|\mathcal{D}_{\nu}|$ number of sites without replacement from $\mathcal{D}$ according to $\left\{ \pi_d^{\mathcal{D}} \right\}$
        \State Construct the feature subset at node $\nu$ based on $\mathcal{D}_{\nu}$ and $\gamma^{\text{c-w}}$: $\bm{z}_{\nu} = \{ \bar{\bm{z}}_{d, m}^{\text{c-w}} \mid d \in \mathcal{D}_{\nu}, m \in \mathcal{\mathcal{M}} \}$
        \State In $\bm{z}_{\nu}$, find the best split using \eqref{eq:node_split_decision_criteria} and \eqref{eq:weighted_pr_nu}
    \EndWhile
    \State Obtain the decision tree $h_{j}^{\text{c-w}}$ and compute its prediction error: $\epsilon_{j}^{\text{c-w}} = \sum_{i = 1}^{n_{\text{train}}} \pi_{\omega_i}^{\Omega}\bm{1}[h_{j}^{\text{c-w}}(\bm{x}_i) \neq y_i]$
    \State Initialize the root of $h_{j}^{\text{s-w}}$ with $\left\{ \left(\bm{x}_i, y_i \right) \right\}$ and $\left\{ \pi_{\omega}^{\Omega} \right\}$
    \While{\textit{splitting criterion is not met for node $\nu$}}
        \State Randomly select $|\mathcal{M}_{\nu}|$ number of categories without replacement from $\mathcal{M}$ according to $\left\{ \pi_m^{\mathcal{M}} \right\}$
        \State Construct the feature subset at node $\nu$ based on $\mathcal{M}_{\nu}$ and $\gamma^{\text{s-w}}$: $\bm{z}_{\nu} = \left\{ \bar{\bm{z}}_{w}^{\text{c-w}} \mid w \in \mathcal{W}_{m}, m \in \mathcal{\mathcal{M}_{\nu}} \right\}$
        \State In $\bm{z}_{\nu}$, find the best split using \eqref{eq:node_split_decision_criteria} and \eqref{eq:weighted_pr_nu}
    \EndWhile
    \State Obtain the decision tree $h_{j}^{\text{s-w}}$ and compute its prediction error: $\epsilon_{j}^{\text{s-w}} = \sum_{i = 1}^{n_{\text{train}}} \pi_{\omega_i}^{\Omega}\bm{1}[h_{j}^{\text{s-w}}(\bm{x}_i) \neq y_i]$
    \State \textbf{If} $\epsilon_j^{\text{s-w}} < \epsilon_j^{\text{c-w}}$: $h_j \leftarrow h_j^{\text{s-w}}$; \textbf{ Else} $h_j \leftarrow h_j^{\text{c-w}}$
    \State For each $i = 1, \cdots, n_{\text{train}}$: 
    \Statex
        \begin{equation}
            \psi_j(\bm{x}_{i}) = \frac{1}{2} \ln{\left(\frac{\text{Pr}_j(y = 1 | \bm{x}_i)}{1 - \text{Pr}_j(y = 1 | \bm{x}_i)}\right)} \label{eq:psi}
        \end{equation}
    \Statex
        \begin{equation*}
            \pi_{\omega_i}^{\Omega} \leftarrow \frac{\pi_{\omega_i}^{\Omega} \exp{\left( -y_i \psi_j\left(\bm{x}_i\right) \right)}}{\sum_{i = 1}^{n_{\text{train}}} \pi_{\omega_i}^{\Omega} \exp{\left( -y_i \psi_j\left(\bm{x}_i\right) \right)}}
        \end{equation*}
\EndFor
\State {\textbf{return}} $\mathcal{H}$, $\Psi$
\end{algorithmic}
\end{algorithm}

\section{Reliability-aware DER Control}\label{sec:opt}

We build on the distribution system model in \cite{zhang2025towards} and a cost reliability model (CRM) that co-optimizes operating cost and EENS, which depends on the ambient- and decision-dependent component failure rates. We then extend this to a multi-source cost reliability model (MCRM) that incorporates meteorological covariates from multiple gridded weather sources via WMSDT and TCSMSB, which map operating conditions and covariates to component risks. Relative to \cite{zhang2025towards}, we also incorporate photovoltaic (PV) units with active power capability. Section~\ref{subsec:rel_integration} describes the model, and Section~\ref{subsec:comp_approach} presents the computational approach.

\subsection{Reliability-aware active distribution grid model} \label{subsec:rel_integration}

The objective function of the reliability-aware model is given as:
\begin{align}
    Obj: \min_{\bm{v}} \left[ \sum_{t \in \mathcal{T}} \mathcal{C}^{\text{op}}_{t}(\bm{v}) + \sum_{t \in \mathcal{T}} \mathcal{C}_{t}^{\text{eens}}(\bm{v}) \right],
\label{eq:CRM_objective}
\end{align}
with
\begin{align*}
    \mathcal{C}_{t}^{\text{eens}}(\bm{v}) &= \left(\Theta_{0, t}^{\text{b}}(\bm{v}) \text{Pr}_{0, t}^{\text{b}}(\bm{v}) \right) \\
    &+\!\!\! \sum_{i \in \mathcal{N}^{+}}\!\! \Theta_{i, t}^{\text{b}}(\bm{v}) \!\! \left(\!\!1\! - \!\left[ 1 - \text{Pr}_{i,t}^{\text{b}}(\bm{v}) \right]\!\!\!\!\! \prod_{j \in \text{MCS}_{b_i}} \!\!\!\!\!\left[ 1 \! - \! \text{Pr}_{j, t}^{\text{l}}(\bm{v}) \right]\!\!\right), \label{eq:reli_measure}
\end{align*}
and
\begin{align*}
    \Theta_{0, t}^{\text{b}} &= \theta_{\text{b}_0}  p_{0, t}, \\
    \Theta_{i, t}^{\text{b}} &= \theta_{i}^{\text{c}} p_{i, t}^{\text{c}} \! + \! \theta_{i}^{\text{PV}} p_{i, t}^{\text{PV}} \! + \! \theta_{i}^{\text{DG}} p_{i, t}^{\text{DG}}\!  +  \! \theta_{i}^{\text{B,c}} p_{i, t}^{\text{B,c}} \! + \! \theta_{i}^{\text{B,d}} p_{i, t}^{\text{B,d}} \! + \! \theta_{i}^{\text{DR}} p_{i, t}^{\text{DR}}, 
\end{align*}
where $\mathcal{T}$ represents the set of modeled time steps, indexed by $t$; $\bm{v}$ denotes the vector of all decision variables (e.g., including power production of distributed generation at bus $i$, time step $t$: $p_{i,t}^{\rm DG}$, solar power set points: $p_{i,t}^{\rm PV}$, battery charging and discharging: $p_{i,t}^{\rm B,c}$, $p_{i,t}^{\rm B,d}$, demand response: $p_{i,t}^{\rm DR}$, and fixed consumption: $p_{i,t}^{\rm c}$). 
We write
$\rm{Pr}_{\textit{i,t}}^{\rm b}$, $\rm{Pr}_{\textit{i,t}}^{\rm l}$ for the interval unreliability (i.e., the probability that a bus or line will fail in time step $t$), $\Theta_{i, t}^{\rm b}$ is the total resulting cost in the event of a failure, and $\theta_{i, t}$ are the individual cost of de-energizing loads or resources.
We define the interval unreliability of the substation (indexed by $0$), buses (indexed by $i$) and lines (indexed by their downstream bus $i$) through the logistic regression model:
\begin{itemize}
    \item Interval unreliability of substation:
        \begin{equation*}
            \text{Pr}_{0,\textit{t}}^{\rm b} = \frac{1}{1 + \lambda_{\text{b}, 0}  e^{-\left(\beta_{0}^{\rm b} \left|p_{0, t}\right|\right)}}. \label{eq:fail_sub}
        \end{equation*}
    \item Interval unreliability of bus $i$, $i \in \mathcal{N}^{+}$:
        \begin{align*}
            \text{Pr}_{i, t}^{\text{b}} &= \frac{1}{1 + \lambda_{\text{b}, i}  e^{-\left(\beta_{i}^{\text{b}} \left|- p_{i, t}^{\text{c}} + p_{i, t}^{\text{DG}} + p_{i, t}^{\text{PV}} - p_{i, t}^{\text{B,c}} + p_{i, t}^{\text{B,d}} + p_{i, t}^{\text{DR}}\right| \right)}} \\
            &= \frac{1}{1 + \lambda_{\text{b}, 0} e^{-(\beta_0^{\text{b}}\tilde{p}_{i, t})}}. \label{eq:fail_bus}
        \end{align*}
    \item Interval unreliability of line $i$, $i \in \mathcal{N}^{+}$:
        \begin{equation*}
            \text{Pr}_{\textit{i,t}}^{\text{l}} = \frac{1}{1 + \lambda_{\text{l}, i}  e^{-\left(\beta_{i}^{\text{l}} l_{i, t} \right)}}. \label{eq:fail_line}
        \end{equation*}
\end{itemize}
We adopt the constraints of MCRM from \cite[Eqs.(1)-(20)]{zhang2025towards} and summarize them below:
\begin{align*}
    &\text{Thermal bounds and power flow relations: \cite[Eqs.(1)-(4)]{zhang2025towards},} \\
    &\text{Nodal power balance and voltage limits: \cite[Eqs.(5)-(10)]{zhang2025towards},} \\
    &\text{Substation power flow limits: \cite[Eqs.(11)-(12)]{zhang2025towards},} \\
    &\text{DER limits with battery state of charge: \cite[Eqs.(13)-(20)]{zhang2025towards}.}
\end{align*}

\subsection{Computational approach} \label{subsec:comp_approach}
Solving the MCRM involves two main challenges: the objective function's non-linear, non-convex nature and the integration of tree-based probability estimates into a logistic framework. To address the first, we apply the sequential convex programming (SCP) approach for non-convex optimization from \cite{zhang2025towards}. For the second challenge, the non-parametric ensemble predictions in Section \ref{sec:tree_ensemble} are cast into a parametric form via logistic regression, ensuring that the SCP of \cite{zhang2025towards} remains valid. Because the MCRM is solved repeatedly, the logistic model is re-parameterized as needed if the solution moves beyond the range underlying the current approximation. Under this methodology, ensemble-predicted failure probabilities serve as response, while power flow information (net power demand for buses and current square magnitude for lines) act as regressor for the fitted model. In this section, we formalize the necessary iterative procedure for approximating the ensemble-generated probabilities via a logistic regression fit. Define $\mathcal{V}_{\text{sub}}^{[k]}$ as the set of power-flow information solutions obtained in iteration $k$. Specifically, for $i \in \mathcal{N}^+$, $p_{0, t}$ denotes the substation power flow, $p_{i, t}$ denotes the net power demand at each bus, and $l_{i, t}$ denotes the current square magnitude at each line. Initially,
\begin{equation*}
    \mathcal{V}_{i}^{[0]} = \left\{ v_{i, t}^{[0]} \mid t \in \mathcal{T} \right\}
\end{equation*}
where $v_{i, t}^{[k]}$ denotes the power flow solution of component $i$ at time $t$ during iteration $k$. After the $k^{th}$ iteration, 
\begin{equation*}
    \mathcal{V}_{i}^{[k]} = \mathcal{V}_{\text{i}}^{[k-1]} \bigcup \left\{ v_{i, t}^{[k]} \mid t \in \mathcal{T} \right\}
\end{equation*}
The process of fitting the logistic regression after the $k^{th}$ is the following.
First, we form a sampling range after the $k^{th}$ iteration for each component $i$ as
\begin{equation*}
    \mathcal{SR}_i^{[k]} = \left[ \min_{t \in \mathcal{T}} \left(v_{i, t}^{[k]}\right) - 2\sigma_{\mathcal{V}_{i}^{[k]}}, \text{ } \max_{t \in \mathcal{T}} \left(v_{i, t}^{[k]} \right) + 2\sigma_{\mathcal{V}_{i}^{[k]}} \right].
\end{equation*}
We then uniformly randomly sample $n_{\text{reg}}$ points of $v_i$ from the interval $\mathcal{SR}_i^{[k]}$, thus obtaining
\begin{equation*}
    \mathcal{V}_{i}^{\text{reg}} = \left\{ v_{i, j}^{\text{reg}} \mid v_{i, j}^{\text{reg}} \in  \mathcal{SR}_i^{[k]} , j = 1, \cdots, n_{\text{reg}}\right\}
\end{equation*}
as depicted in Fig.~\ref{fig:sampling_range}.
\begin{figure}
    \centering
    \resizebox{\linewidth}{!}{\begin{tikzpicture}
    \draw[dotted, thick, blue!50] (0,0) -- (3,0) ;
    \draw[dotted, thick, blue!50] (3,0) -- (8,0) ;
    \draw[->, dotted, thick, blue!50] (8,0) -- (11,0) ;

    \draw[->, thick, darkgreen!50] (0, -0.5) -- (11, -0.5);

    \foreach \x [count=\i] in {3.4, 4.1, 4.8, 6.1, 6.4, 7.4} {
        \node[circle, draw=blue, fill=blue!20, inner sep=2pt] (circle\i) at (\x,0) {};
    }

    \foreach \x [count=\i] in {3.0, 8.0} {
        \node[circle, draw=blue, fill=blue!60, inner sep=2pt] (circle\i) at (\x,0) {};
    }

    \draw[->, thick, blue] (3.0,0.2) -- ++(0,0.5)
       node[xshift=1pt, yshift=7pt] {\footnotesize $\min_{t \in \mathcal{T}} \left(v_{i, t}^{[k]} \right)$};
       
    \draw[->, thick, blue!50] (4.8,0.2) -- ++(0,0.5)
       node[xshift=1pt, yshift=7pt] {\footnotesize $v_{i, t}^{[k]}$};

    \draw[->, thick, blue] (8.0,0.2) -- ++(0,0.5)
       node[xshift=1pt, yshift=7pt] {\footnotesize $\max_{t \in \mathcal{T}} \left(v_{i, t}^{[k]} \right)$};

    \foreach \x [count=\j] in {1.0, 10} {
        \node[draw=darkgreen, fill=darkgreen!60, regular polygon, regular polygon sides=3, shape border rotate=0, minimum size=6pt, inner sep=0pt, yscale=1.8] (triangle\j) at (\x,-0.7) {};
    }

    \draw[->, thick, darkgreen] (1.0,-1) -- ++(0,-0.5)
       node[xshift=1pt, yshift=-10pt] {\footnotesize $\min_{t \in \mathcal{T}} \left(v_{i, t}^{[k]} \right) - 2 \sigma_{\mathcal{V}_i^{[k]}}$};

    \draw[->, thick, darkgreen!60] (5.5,-1) -- ++(0,-0.5)
       node[xshift=1pt, yshift=-9pt] {\footnotesize $v_{i, j}^{\text{reg}}$};

    \draw[->, thick, darkgreen] (10,-1) -- ++(0,-0.5)
       node[xshift=1pt, yshift=-10pt] {\footnotesize $\max_{t \in \mathcal{T}} \left(v_{i, t}^{[k]} \right) + 2 \sigma_{\mathcal{V}_i^{[k]}}$};
    
    \foreach \x [count=\j] in {2.4, 5.5, 7.7} {
        \node[draw=darkgreen, fill=darkgreen!30, regular polygon, regular polygon sides=3, shape border rotate=0, minimum size=6pt, inner sep=0pt, yscale=1.8] (triangle\j) at (\x,-0.7) {};
    }

\end{tikzpicture}}
    \vspace{-1.5em}
    \caption{Illustration of the sampling interval $\mathcal{SR}_i^{[k]}$.}
    \label{fig:sampling_range}
\end{figure}
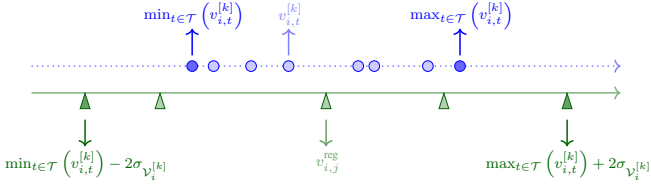

We combine each sampled point $v_{i, j}$ of component $i$ with the known weather data, and input these regressors into the trained ensemble model detailed in Section \ref{sec:tree_ensemble}. 
This yields a predicted failure probability for component $i$ after the $k^{th}$ iteration. 
We then define a single logistic model for component $i$ across all time steps after the $k^{th}$ iteration by fitting a logistic regression to the paired data 
\begin{equation*}
    \left\{ \left(v_{i, j}^{\text{reg}}, \hat{\text{P}}\text{r}_{i, j} \right) \mid j = 1, \cdots, n_{\text{reg}} \right\}.
\end{equation*}
The iterative steps and the conditions for updating each component's logistic regression parameters are detailed in Algorithm \ref{alg:update_reg_model}. The complete procedure for solving MCRM with SCP is provided in Algorithm~\ref{alg:complete_CRM_alg} in \ref{appendix:complete_CRM_alg}.

\begin{algorithm}
\caption{Updating Regression Model for $k^{th}$ Iteration}
\label{alg:update_reg_model}
\begin{algorithmic}[1]
    \If{$v_{i, t}^{[k]} \notin \mathcal{SR}_{i}^{[k]}, \forall i \in \mathcal{N}, t \in \mathcal{T}$} 
        \For{each component $i$}
            \State Form $\mathcal{SR}_{i}^{[k]}$
            \State Obtain $\mathcal{V}_{i}^{\text{reg}}$
            \State Fit the logistic regression model for component $i$
            \State Update the parameter for component $i$, $\forall t \in \mathcal{T}$
            \State Re-solve MCRM
        \EndFor
    \EndIf
\end{algorithmic}
\end{algorithm}

\section{Case Study} \label{sec:case_study}

\subsection{Experiment design} \label{subsec:experiment}
Our experiment design addresses gaps between theoretical models and real-world applications. 
Component-wise failure data from real power systems are typically not publicly available, creating a challenge in obtaining a dataset tailored exactly to our research setting \cite{richard_epdr_4, zhang2025towards}. 
Therefore, we adopt a semi-synthetic experimental approach. This method involves using credible real-world datasets if available and, in cases where required data is unavailable. 
Practitioners with access to real failure and/or outage data can plug their available datasets into our method to replace the synthetic data.

\subsubsection{Power system and weather data} \label{subsubsec:system_data} 
We select an exemplary substation location at latitude $40.55^{\circ}$N and longitude $-74.34^{\circ}$W. 
Our analysis focuses on a single day of operating an active distribution system. We adopt a modified version of the radial IEEE $33$-bus test feeder as detailed in \cite{zhang2025towards} with additional PV units. 
Fig.~\ref{fig:33_bus_power_radial_system} in \ref{appendix:sys_layout} illustrates the system topology and shows the locations of DERs. The system's load profiles are obtained from the procedure outlined in \cite{zhang2025towards} \cite{thurner2018pandapower, nyiso} from publicly available load data. 
To add adequate PV capacity, we first calculate the system's peak load, defined as the maximum total load across all buses:
\begin{equation*}
    p^{\text{c}}_{\text{total, peak}} = \max_{t \in \mathcal{T}}\left( \sum_{i \in \mathcal{N}^+} p_{i, t}^{\text{c}} \right).
\end{equation*}
We define two penetration levels for the total PV capacity: a low penetration level at $5\%$ and a high penetration level at $50\%$ of the peak load:
\begin{equation*}
    p_{(5\%), \text{total}}^{\text{PV}} = 0.05 \times p^{\text{c}}_{\text{total,peak}} \text{ }, \quad p_{(50\%), \text{total}}^{\text{PV}} = 0.50 \times p^{\text{c}}_{\text{total,peak}}.
\end{equation*}

The selection of locations for PV installations is guided by the existing presence of DG and BESS units. 
Co-locating multiple DERs at the same bus simplifies operational coordination and potentially reduces additional infrastructure costs. Furthermore, buses already equipped with DG or BESS generally possess the necessary interconnection capacity and infrastructure, facilitating the integration of new PV systems without significant rewiring or infrastructure upgrades \cite{chinaris2025hybridization}. 
In the high PV penetration scenario, buses equipped with either a DG or BESS units receive PV installations. In the low PV scenario, only buses equipped with both DG and BESS units are selected. To assign the rated power of the PV inverter at each selected bus, we leverage the fact that a bus already hosting a significant amount of dispatchable or storage capacity can inherently accommodate a comparable level of solar generation \cite{lu2018multiobjective, bekele2010feasibility}.
For each bus $i \in \mathcal{N}_{\text{PV}}$, we set each bus' nominal PV capacity to be
\begin{equation*}
    p_{\text{rated},i}^{\text{PV}} = \frac{\max ( P_{i}^{\text{DG},\max}, P_{i}^{\text{BESS},\max} )}{\sum_{j \in \mathcal{N}^{\text{PV}}} \max ( P_{j}^{\text{DG},\max}, P_{j}^{\text{BESS},\max} )} \times p^{\text{PV}}_{(*), \text{total}}.
\end{equation*}

We obtain the baseline solar generation profile from \cite{nrel}. 
We select a standard $4$kW DC-rated system with the default DC-to-AC ratio of $1.2$, which yields an inverter rating of approximately $3.33$ kW AC \cite{nrel}. 
To appropriately scale the PV active power injection ($p_{i, t}^{\text{PV}}$) at each bus $i$, $i \in \mathcal{N}_{\text{PV}}$ to fit our system, we first identify the peak power in the reference profile as $p_{\text{pv original, peak}} = \max_{t \in \mathcal{T}}\left( p_{\text{pv original}, \text{ } t} \right)$, then 
\begin{equation*}
    p_{i, t}^{\text{PV}} = p_{\text{rated},i}^{\text{PV}} \times \frac{p_{\text{pv original}, \text{ } t}}{p_{\text{pv original, peak}}}.
\end{equation*}

To capture meteorological variables near our substation location, we use the weather dataset from the Copernicus service \cite{copernicus}. It provides gridded data in $0.25^{\circ}$ increments of latitude and longitude, and we identify the four nearest points that form a rectangle around the station (i.e., $D = 4$ measurement sites). The latitudes and longitudes of these weather datasets are $(40.54^{\circ}, -74.21^{\circ})$, $(40.54^{\circ}, -74.46^{\circ})$, $(40.79^{\circ}, -74.21^{\circ})$, and $(40.79^{\circ}, -74.46^{\circ})$, respectively. For each of these four measurement sets, we extract $18$ weather features (i.e., $W = 18$). In line with the Copernicus data structure, we organize these $18$ features into $4$ broader categories (i.e., $M = 4$, detailed in \ref{appendix:para_values}).

\subsubsection{Data for training, validation and testing} \label{subsubsec:model_train_data}
Due to the scarcity of real-life data, this section provides detailed explanations of how we constructed our ensemble model training data using all currently available datasets. Additionally, we outline how the data should be appropriately utilized when real-life observations become available, specifically addressing the challenge of accurately modeling rare failure events. 

\paragraph{Synthetic regressor data} 
For training, validation, and testing, we require a sufficient number of data points. The standard 33-bus test system dataset supplies only one static load vector \cite{thurner2018pandapower}. We extend it to $122$ days (April $01$, $2024$ to July $31$, $2024$) at a two-hour resolution ($12$ time steps per day) by scaling that snapshot according to the New York City system load obtained from \cite{nyiso}, following the method from \cite{zhang2025towards}. 
We focus on this warm season window because it targets the failure mechanisms that intensify with higher temperatures and peak loads.
(Including winter would also be possible and introduce distinct cold-related mechanisms (e.g., icing) and markedly different load patterns driven by heating demand \cite{panteli2016boosting, panteli2015influence, thornton2016role}.)
Moreover, limiting the horizon to a short period rather than spanning several years allows us to treat component condition as effectively constant, avoiding confounding from general long-term degradation \cite{wang2002reliability}.
In practice, the load series from April to July captures the dominant annual structure with July typically contains the annual system peak and exhibits the largest daily load variability, whereas April reflects cooler season conditions with daily profiles closer to winter \cite{online_eia2020hourly, online_eia2023peak}. After extending the load profile, we then feed the resulting $122 \times 12 = 1464$ distinct load snapshots into the CM model described in \cite{zhang2025towards}, obtaining $1464$ data points corresponding to net power demand (response variable $\mathcal{O}^{\text{b}}$) and another $1464$ data points corresponding to current magnitude (response variable $\mathcal{O}^{\text{l}}$). 
Each of these $1464$ data points for net power demand and current magnitude is paired with the synchronous two-hour weather record, giving $72$ meteorological covariates ($D = 4$, $W = 18$) per time step. 

To ensure a sufficient number of observations for training, validation, and testing, we apply a non-parametric bootstrap with replication factor of $20$. The final sample sizes for both buses and lines are $n_{\text{syn}} = 1464 \times 20 = 29280$. 
The predictor matrix is standardized column-wise (zero mean, unit variance). 
Formally:
\begin{align*}
    \bm{X}^{\diamond} &= \big\{ x_{i, d, m, w, o} =  X_{d, g, w, o}\left( \omega_i \right) \mid i = 1, \cdots, n_{\text{syn}}, \\
    d \in& \{1, \cdots, 4\}, m \in \{1, \cdots, 4\}, w \in \{1, \cdots, 18\}, o \in \mathcal{O}^{\diamond} \big\}
\end{align*}
with $\diamond = \text{b}$ for bus and $\diamond = \text{l}$ for line, respectively.

\paragraph{Synthetic response variable data} 
In real-world power system operations, component-level failures are rare \cite{toftaker2023accounting}. 
Detailed failure logs specifically within each two-hour window are generally unavailable. We generate synthetic failure data under the assumption that higher values of specific weather and power flow features correspond to an increased likelihood of component failures, as supported by evidence from numerous studies (e.g., \cite{do2023spatiotemporal, taylor2022statistical, banasik2024influence, souto2024identification, billinton2005consideration, billinton2006distribution, richard_epdr_3}). To assign failure labels ($\mathrm{fail} = 1$ or $\mathrm{non\text{-}fail} = -1$), 
we proceed as follows. 
Let $X_{d,m,w}^{\text{wf}}$ denote the random variable of $X_{d,m,w}^{\text{wf}}: \Omega^{\text{wf}} \rightarrow \mathbb{R}$, where $\Omega^{\text{wf}}$ spans all the weather features. We note that $\bm{X}^{\text{b}}$ and $\bm{X}^{\text{l}}$ share the same set of weather features. The realized dataset for $n_{\text{syn}}$ i.i.d. observations as
\begin{align*}
    \{x_{i,d,m,w}^{\text{wf}} = X_{d,m,w}^{\text{wf}}(\omega_i^{\text{wf}}) \mid i = 1, \cdots, n_{\text{syn}}, \\
    d \in \mathcal{D}, m &\in \mathcal{M}, w \in \mathcal{W}_{m}\},
\end{align*}
where $\{\omega_i^{\text{wf}}\}_{i = 1}^{n_{\diamond}} \subseteq \Omega^{\text{wf}}$. 
For each observation $i \in \{1, \cdots, n_{\text{syn}}\}$, measurement site $d \in \mathcal{D} = \{1, \cdots, 4\}$, and category $m \in \mathcal{M} = \{1, \cdots, 4\}$, we form the site-level category average as
\begin{equation*}
    \bar{x}_{i, d, m}^{\text{wf}} = \frac{1}{|\mathcal{W}_{m}|}\sum_{w \in \mathcal{W}_m}x_{i, d, m, w}^{\text{wf}}.
\end{equation*}
For each measurement site $d$, we calculate the distance-based site weights using Eq.\eqref{eq:site_wise_sample_weight} and we obtain the fixed weights as
\begin{equation*}
    \left( \pi_{1}^{\mathcal{D}}, \pi_{2}^{\mathcal{D}}, \pi_{3}^{\mathcal{D}}, \pi_{4}^{\mathcal{D}} \right) = (0.059, 0.860, 0.036, 0.045).
\end{equation*}
For each observation $i \in \{ 1, \cdots, n_{\text{syn}}\}$ and category $m \in \mathcal{M} = \{0, \cdots, 4\}$, define the site-weighted category average as
\begin{equation*}
    \bar{x}_{i, m}^{\text{wf}} = \sum_{d = 1}^{|\mathcal{D}|} \pi_d^{\mathcal{D}} \bar{x}_{i, d, m}^{\text{wf}}.
\end{equation*}
Let $X^{\text{pf}, \diamond}$ denotes the random variable of $X^{\text{pf}, \diamond}: \Omega^{\text{pf},\diamond} \rightarrow \mathbb{R}$, where $\Omega^{\text{pf}, \diamond}$ span all the power flow information features of $\diamond$, $\diamond = \{\text{b}, \text{l}\}$. The realized datasets for buses and lines each consist of $n_{\text{syn}}$ i.i.d. observations and are
\begin{equation*}
    \{x_{i}^{\text{pf}, \diamond} = X^{\text{pf}, \diamond}(\omega_i^{\text{pf}, \diamond}) \mid i = 1, \cdots, n_{\text{syn}}\}.
\end{equation*}
We define a summary matrix consisting of category-wise weighted means from the weather features combined with power flow information, given as
\begin{equation*}
    \bm{X}_{\text{agg}}^{\diamond} = \left( \bar{x}_{i, 1}^{\text{wf}}, \bar{x}_{i, 2}^{\text{wf}}, \bar{x}_{i, 3}^{\text{wf}}, \bar{x}_{i, 4}^{\text{wf}}, x_i^{\text{pf}, \diamond} \right) \in \mathbb{R}^{n_{\diamond} \times 5}.
\end{equation*}
To generate synthetic failure labels, we first randomly select $5\%$ of all observations and label them as failures. Next, we compute the column-wise means and standard deviations of $\bm{X}_{\text{agg}}^{\diamond}$. For each observation, if any column value exceeds its column mean by more than one standard deviation, we randomly assign failures to $75\%$ of these observations. If an operating point was previously labeled as a failure, this label is retained. We denote the response variable vector as $\bm{Y}^{\diamond}, \diamond = \{\text{b}, \text{l}\}$.

\subsubsection{Ensemble model training and prediction} \label{subsubsec:model_train_pred}

\paragraph{Training, validation, and testing} We stratify the complete sample $\left( \bm{X}^{\diamond}, \bm{Y}^{\diamond} \right)$ by the class label and hold out $20\%$ of the observations as an external test set. The remaining $80\%$ enters a $5$-fold cross validation loop for hyperparameter tuning. For the WMSDT (see Section~\ref{subsubsec:ensemble_model_wmsdt}), we search the grid $|\mathcal{D}_{\nu}| = |\mathcal{M}_{\nu}| \in \{1, 2, 3, 4\}$, $J \in \{50, 60, \cdots, 200\}$, and $H_{\max} \in \{3, 4, \cdots, 10\}$. 
For the TCSMSB (see Section~\ref{subsubsec:ensemble_model_tsmsb}), the additional tuning grid is $\gamma^{\text{c-w}} = \gamma^{\text{s-w}} \in \{60\%, 65\%, \cdots, 100\%\}$. We choose $|\mathcal{D}_{\nu}| = |\mathcal{M}_{\nu}| = 2$, $J = 150$, $H_{\max} = 8$ and $\gamma^{\text{c-w}} = \gamma^{\text{s-w}} = 90\%$ to balance the predictive accuracy and computational cost. 

\paragraph{Component failure probability calibration} 
The component failure probabilities produced by the ensemble classifiers feed into the MCRM, so their numerical accuracy influence every subsequent optimization result. Directly utilizing the raw component failure probabilities predicted by our proposed models based on the generated synthetic data is inappropriate, primarily for two reasons. Firstly, all the trees in our models are grown by minimizing Gini impurity, the learning objective rewards correct assignments but is indifferent to the absolute scale of the scores \cite{fawcett2006introduction, niculescu2005predicting, carreno2020analyzing}. They can output probability estimates whose values significantly deviate from the underlying event frequencies \cite{platt1999probabilistic, he2009learning}. On the other hand, the synthetic data contains a far larger proportion of failures than occur in practice. In this study, we assume prior knowledge of average failure rates for buses ($\text{Pr}_{i}^{\text{b,avg}} = 4.93 \times 10^{-6}, \forall i \in \mathcal{N}^{+}$) and lines ($\text{Pr}_{i}^{\text{l,avg}} = 1.14 \times 10^{-5}, \forall i \in \mathcal{N}^{+}$) per two-hour window, adopted from \cite{zhang2025towards}. In contrast, the proportion of failure cases within our synthetic dataset is approximately $44\%$ for both buses and lines, significantly higher than their actual occurrence rates. To account for the inflated failure rate, we apply a partial global prior calibration that adjusts predicted probabilities to match the known priors within the central probability region, while preserving the ranking of individual components. We set
\begin{equation}
    \text{Pr}_i^{\diamond} = 
    \begin{cases}
        &\frac{\hat{\text{P}}\text{r}_i^{\diamond, \text{raw}}\text{Pr}_i^{\diamond, \text{avg}}\frac{n^{\diamond}_{\text{syn}}}{n_{\text{syn,fail}}}}{\hat{\text{P}}\text{r}_i^{\diamond, \text{raw}}\text{Pr}_i^{\diamond, \text{avg}} \frac{n_{\text{syn}}}{n^{\diamond}_{\text{syn,fail}}} + (1 - \hat{\text{P}}\text{r}_i^{\diamond, \text{raw}})(1 - \text{Pr}_i^{\diamond, \text{avg}}) \frac{n_{\text{syn}}}{n_{\text{syn}} - n^{\diamond}_{\text{syn,fail}}}} \\
        &\sixteenquad \quad \quad \quad \hat{\text{P}}\text{r}_i^{\diamond,\text{raw}}\in \mathcal{CI}, \\
        &\hat{\text{P}}\text{r}_i^{\diamond,\text{raw}} \sixteenquad \quad \halfquad \text{otherwise}.
    \end{cases}
    \label{eq:global_prior_prob_calibration}
\end{equation}
where $\mathcal{CI} = [0.05, 0.95]$; $n_{\text{syn}, \text{fail}}^{\diamond}$ represents the number of failure instances for component type $\diamond$, out of a total of $n_{\text{syn}}$ observations, with $\diamond \in \{\text{b}, \text{l}\}$. \textit{W.l.o.g.}, the central probability region is set to be $[0.05, 0.95]$. We provide detailed justification of Eq.\eqref{eq:global_prior_prob_calibration} in \ref{appendix:global_prior_prob_calibration}.

\begin{remark}
    We apply global prior calibration because field observations of individual failures are unavailable. 
    When field failure data is available, however, directly fitting our proposed models to the raw real-life data and injecting the resulting probabilities into the MCRM is ill-advised. 
    Two statistical obstacles arise. Firstly, because of the inherently low failure rate, an excessively large sample size would be necessary to capture a sufficient number of failure instances, leading to computational inefficiency. Secondly, in samples characterized by an extremely low ratio of failure to non-failure cases, standard learners, including our proposed ensembles, optimize overall accuracy by favoring the majority class and consequently produce posterior probabilities that are miscalibrated for the minority class, as discussed in \cite{king2001logistic, he2009learning, ghosh2024class}. 
    A practical approach to addressing this issue involves restructuring the available dataset to balance the proportions between failure and non-failure cases. Common strategies include oversampling existing failure instances or synthetically generating additional failure cases using techniques such as SMOTE \cite{he2009learning, chawla2002smote}. As in our generated synthetic data case, following dataset restructuring, probability calibration is typically applied. When observations of individual failures are accessible, more sophisticated calibration methods, such as Platt scaling or isotonic regression, can be applied to accurately adjust predicted probabilities to match empirical failure rates \cite{platt1999probabilistic, guo2017calibration}.
\end{remark}

\subsection{Result and discussion} \label{subsec:result}

\subsubsection{Ensemble models} \label{table:ensemble_models_compare}
Table~\ref{table:ensemble_models_compare} summarizes the performance metrics for the proposed models (highlighted in grey) and the benchmark models: linear regression (LR; parametric linear model), decision tree (DT; non-parametric single tree), RF (ensemble of bagged trees), AdaBoost (ensemble of boosted trees), and SVM (large margin classifier). Here, accuracy is computed on the synthetic dataset as the proportion of instances for which the predicted label equals the observed label, i.e., $\text{accuracy} = \frac{1}{n}\sum_{i = 1}^{n_{\text{syn}}}\bm{1}\{\hat{y}_i = y_i\}$. Specifically, under identical conditions, the runtime for AdaBoost is $36.57$ seconds, whereas the runtime for TCSMSB is $22.12$ seconds.
\begin{table}[ht]
\centering
\begin{tabular}{lcccc}
\toprule
& MSE & Accuracy ($\%$) & Parameter \\
\midrule
LR        & $0.1691$ & $75.05$ & - \\
DT        & $0.1627$ & $75.87$ & $H_{\max} = 8$ \\
RF        & $0.1440$ & $76.56$ & $H_{\max} = 8$, $J = 150$ \\
  &  &  & $\text{max features} = 100\%$ \\
SVM       & $0.1238$ & $82.75$ & kernel = rbf \\
AdaBoost  & $0.1379$ & $81.00$ & $J = 150$ \\
  &  &  & base learner = WMSDT \\
\rowcolor{gray!20} 
WMSDTE    & $0.1022$ & $86.72$ & See Section \ref{subsubsec:model_train_pred} \\
\rowcolor{gray!20} 
TCSMSB   & $0.1759$ & $81.10$ & See Section \ref{subsubsec:model_train_pred} \\
\bottomrule
\end{tabular}
\caption{Performance comparison across models; highlighted rows corresponding to the approaches proposed in this paper}
\end{table}

\subsubsection{System operations}
\label{subsubsec:result_mcrm}

Figure~\ref{fig:con_obj_low} (Figure~\ref{fig:con_obj_high}) comprises $5$ subplots illustrating the objective function of the MCRM under the low (high) PV scenario. Subplot (i) presents the convergence behavior of the algorithm, featuring four curves representing original and linearized objective values from WMSDTE and TCSMSB, respectively. For all scenarios, the MCRM consistently converges within $40$ iterations, effectively reducing the objective by approximately $16.04\%$. Subplot (ii) isolates the operational cost component of the objective function across iterations, comparing outcomes from WMSDTE and TCSMSB. Subplot (iii) specifically addresses the EENS, which is subject to linearization during the optimization process. Two curves from each statistical model are displayed, representing original and linearized values. Subplot (iv) and subplot (v) illustrate the distribution of the final objective values derived through simulations. Subplot (iv) presents the distribution of the objective values obtained from a Monte Carlo simulation. 

The trained WMSDTE and TCSMSB models predict each bus and line's failure probabilities based on the MCRM's solutions. For each of the $1 \times 10^{5}$ trials, component failures are simulated using independent Bernoulli trials. The objective values are calculated using each simulated failure scenario. Subplot (v) employs an alternative validation methodology to assess the optimized solutions. Failure scenarios are directly drawn from synthetic datasets initially generated for training the component failure prediction models. To ensure the realism of the scenarios, these synthetic failure rates are scaled to match actual, very low, component failure probabilities. Decision variables form the MCRM's solutions (for both WMSDTE and TCSMSB) are applied to sampled failure cases to compute the corresponding objective values. This validation procedure includes only those trials featuring at least one component failure, yields a higher objective values to show the potential economic consequences under operational failure conditions. 
We highlight that the WMSDTE model yields lower objective values accompanied by greater variance compared to the TCSMSB model, which produces higher average objective values but lower variance. 
Hence, system operators may select the model appropriate to their risk preferences.

Figure~\ref{fig:der_utilization} shows DER operation the CM and MCRM model for the high PV scenario under the WMSDTE prediction model. 
The situations for low PV and using TCSMSB prediction model are qualitatively similar. 
In Figure~\ref{fig:der_utilization}(a) only some BESS is used to utilize low demand and low electricity prices. 
Other DERs remain unused in this case study because the system can remain within its safe operating conditions only using comparatively cheap power from the substation.
When we make the model reliability-aware, however, we observe that because of value added from reducing the expected cost of component failures, MCRM utilizes available  DG and DR more, with DG output increasing during the late-day peak. 

\begin{figure}
    \centering    \includegraphics[width=1.0\linewidth]{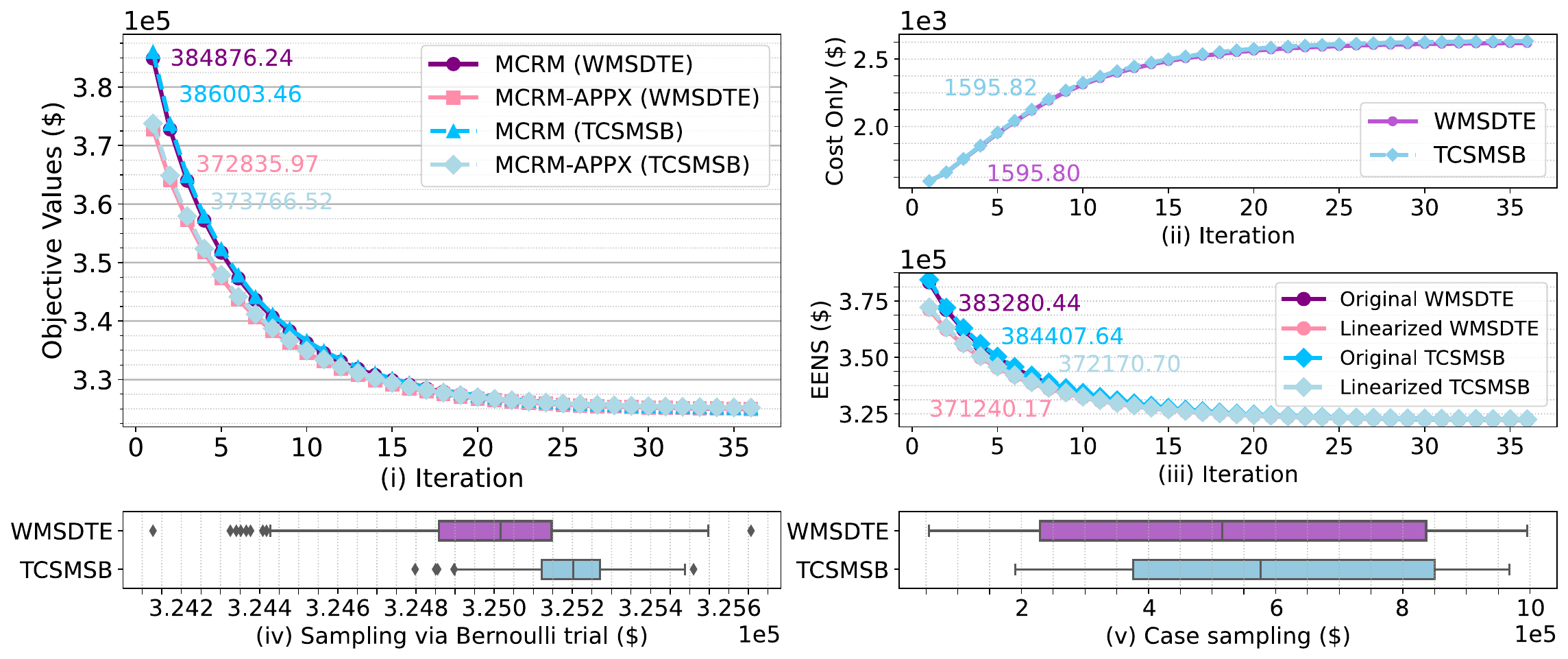}
    \caption{Convergence and objective function (low PV)}
    \label{fig:con_obj_low}
\end{figure}

\begin{figure}
    \centering    
    \includegraphics[width=1.0\linewidth]{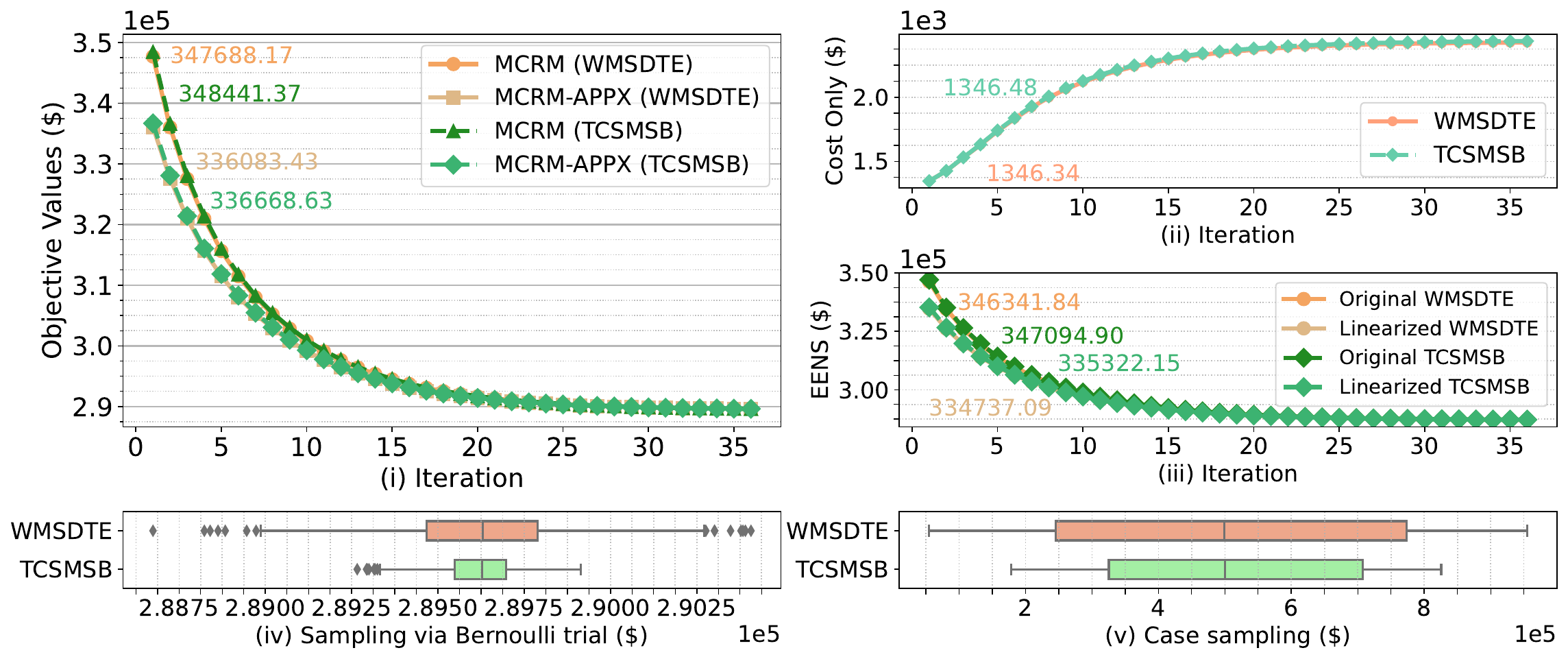}
    \caption{Convergence and objective function (high PV)}
    \label{fig:con_obj_high}
\end{figure}

\begin{figure}
    \centering
    \begin{tabular}{@{}c@{\hspace{1em}}c@{}}
        \subcaptionbox{CM\label{fig:der_utilization_cm}}{%
            \includegraphics[height=0.21\textwidth]{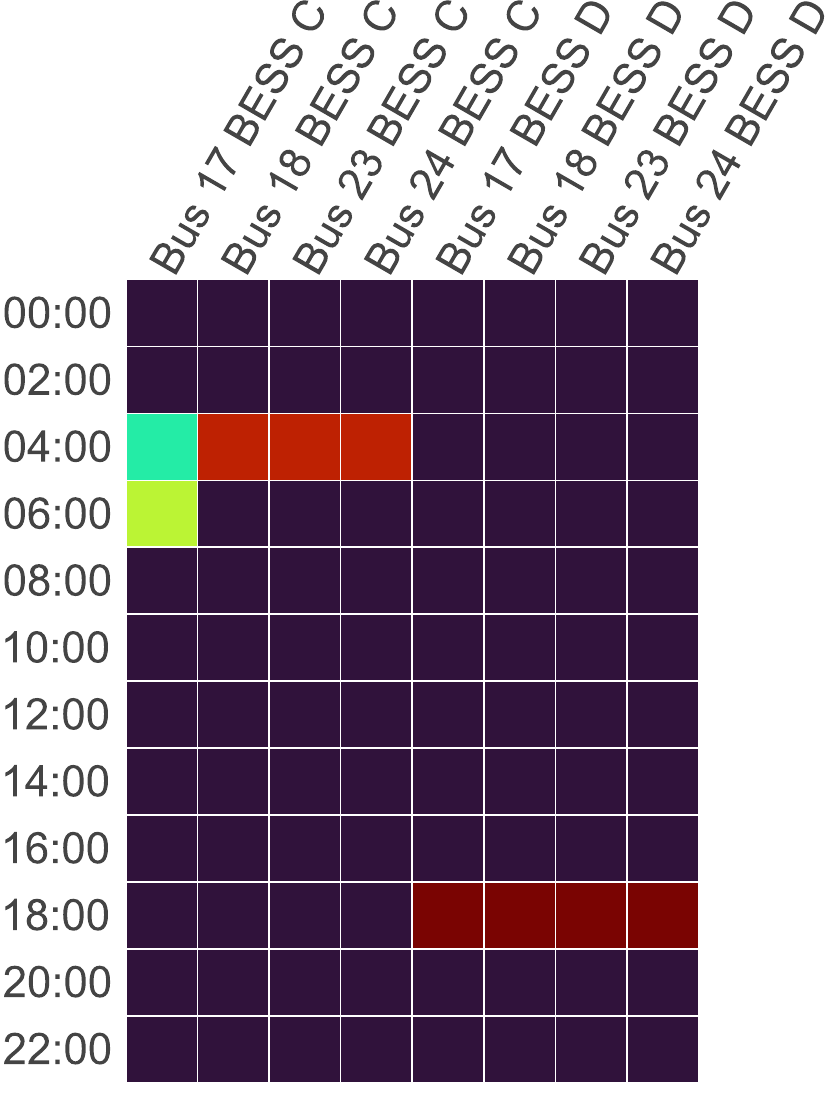}%
        }%
        &
        \subcaptionbox{MCRM\label{fig:der_utilization_mcrm}}{%
            \includegraphics[height=0.21\textwidth]{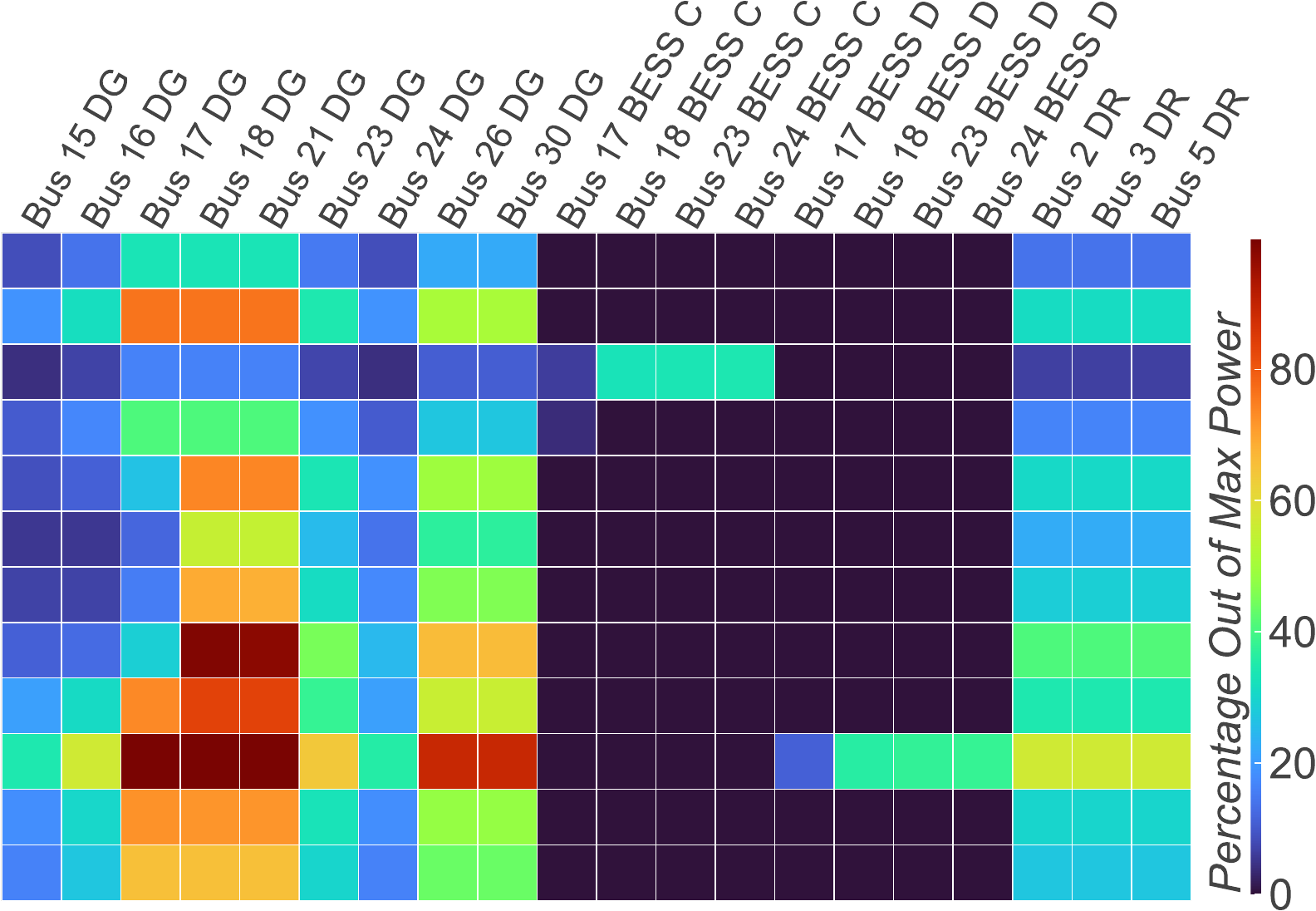}%
        }%
    \end{tabular}
    \caption{DER active power dispatch in percent utilization of maximum power limit resulting from CM and MCRM in the ``high PV'' scenario. In CM, DG and DR are inactive. ``BESS C'' denotes battery charging; ``BESS D'' denotes battery discharging.}
    \label{fig:der_utilization}
\end{figure}

\section{Conclusion} \label{sec:conclusion}
In this study, we incorporated decision-dependent system reliability into the power flow optimization problem, allowing us to simultaneously minimize system operational costs and enhance reliability. Our approach accounts for reliability influenced by both internal decision variables and external conditions. To effectively manage multi-dimensional external conditions, we propose two ensemble models designed for real world scenarios involving available meteorological data. Finally, we validated the effectiveness of our iterative solution approach in addressing the inherent non-convexity of the optimization problem. 

\appendices

\section{Active distribution grid model constraints} \label{appendix:constraints_math_mode} 
Relative to the constraint in \cite{zhang2025towards}, we replace \cite[Eq.(7)]{zhang2025towards}, with the following to incorporate PV: $\forall i \in \mathcal{N}^+, \forall t \in \mathcal{T}$:
\begin{equation*}
    f_{i, t}^p = \sum_{j \in \mathcal{S}c_i}\left( f_{j, t}^p + l_{j, t} R_{j}^{\text{l}} \right) + \sum_{\diamond_{\text{in}}}p_{i, t}^{\diamond_{\text{in}}} - \sum_{\diamond_{\text{out}}}p_{i, t}^{{\diamond_{\text{out}}}} + G_{i}v_{i, t},
\end{equation*}
where $\diamond_{\text{in}} = \{\text{(c)}, \text{(B,c)}\}$ and $\diamond_{\text{out}} = \{\text{(DG)}, \text{(PV)}, \text{(B,d)}, \text{(DR)}\}$.

\section{Proof of Proposition \ref{prop:small_mse_targeted_samp}} \label{appendix:benefits_weighted_feature_selection}
Under the assumptions that $|\mathcal{I}_{\rm{ter,n}}| \rightarrow \infty$ and $|\mathcal{I}_{\rm{ter,n}}|/n \rightarrow 0$, the conditions of Theorem~\ref{thm:rf_mse_bound} are satisfied, ensuring consistency of the RF estimator \cite{biau2012analysis, scornet2015consistency}. We first note that Theorem~\ref{thm:rf_mse_bound} is non-asymptotic. The inequality \eqref{eq:mse_total} holds for every sample size $n$ and for every collection of $\{ \xi_{n, j}\}_{j = 1}^{\kappa}$ such that $\{ \pi_{n, j}^{\Sigma} \}_{j = 1}^{\kappa}$ in \eqref{eq:coor_gen_samp_dist} defines a probability vector. In the case of uniform sampling $\pi_{n, j}^{\Sigma} = 1/\kappa, \forall n,j$:
\begin{equation*}
    \xi_{n, j}^{\rm{us}} = \begin{cases}
        (\varsigma-\kappa)/\kappa, & \text{if } j \in \Sigma \\
        1/\kappa, & \text{otherwise},
    \end{cases}
\end{equation*}
In the case of $\delta$-targeted sampling:
\begin{equation*}
    \xi_{n, j}^{\rm{ts}} = \begin{cases}
        (\delta_{n,j}\varsigma-\kappa)/\kappa, & \text{if } j \in \Sigma \\
        (1 - \sum_{j\in\Sigma}\frac{\delta_{n,j}}{\kappa})/(\kappa - \varsigma), & \text{otherwise},
    \end{cases}
\end{equation*}
We only need to consider the values of $\xi_{n,j}$ for $j \in \Sigma$ since $\xi_n$ only depends on the sequences $\{ (\xi_{n, j}: j \in \varsigma) \}$. We notice that in \eqref{eq:mse_variance_component}:
\begin{equation} \label{eq:component_xi_n}
    (1 + \xi_{n,j})^{-1}\left( 1 - \frac{\xi_{n,j}}{\varsigma - 1}\right)^{-1} = \frac{\varsigma - 1}{\varsigma - 1 + \xi_{n,j}(- \xi_{n, j} + \varsigma - 2)},
\end{equation}
which is monotonically decreasing w.r.t. $\xi_{n, j}$ when $\xi_{n, j} < (\varsigma-2)/2 \Leftrightarrow \delta_{n,i} < \kappa/2$ and \eqref{eq:component_xi_n} is strictly positive $\forall j \in \Sigma$. The condition $\delta_{n, i} < \kappa/2$ follows directly from determining the axis of symmetry of the parabola w.r.t. $\xi_{n, i}$. The strict positivity is justified as follows. Under the constraints of $\delta_{n,j}$, the boundary of $\xi_{n, j}$ is given as:
\begin{equation*}
    \frac{\varsigma - \kappa}{\kappa} < \xi_{n, j} < \frac{\frac{\kappa}{2}\varsigma - \kappa}{\kappa}, \quad \forall j\in \Sigma.
    \end{equation*}
The minimum value of the denominator of \eqref{eq:component_xi_n} is strictly larger than both of:
\begin{align}
    & \varsigma - 1 + \frac{\varsigma - \kappa}{\kappa}(- \frac{\varsigma - \kappa}{\kappa} + \varsigma - 2); \label{eq:mse_proof_positive_check_1} \\
    & \varsigma - 1 + \frac{\frac{\kappa}{2}\varsigma - \kappa}{\kappa}(- \frac{\frac{\kappa}{2}\varsigma - \kappa}{\kappa} + \varsigma - 2) \label{eq:mse_proof_positive_check_2}
\end{align}
An algebra check confirms that both \eqref{eq:mse_proof_positive_check_1} and \eqref{eq:mse_proof_positive_check_2} are strictly positive values. Thus, under this condition, $\xi_n$ is monotonically decreasing w.r.t. $\xi_{n, j}, \forall j \in \Sigma$ and $\Xi_n$ is monotonically decreasing w.r.t. $\xi_{n, j}$. Therefore, the first term of the MSE bound in \eqref{eq:mse_total} is monotonically decreasing w.r.t. $\xi_{n,j}$. We also note that the second half of the MSE bound in Eq.\eqref{eq:mse_total} is a monotonically decreasing function w.r.t. $\varpi_n$ and $\varpi_n^{\rm{us}} < \varpi_n^{\rm{ts}}, \forall n$. Thus, it will have strictly smaller values in the case of $\delta$-targeted sampling compared to uniform sampling, where the strict inequality follows from the fact that $\exists j \in \Sigma, \delta_{n, j} > 1$ so that $\xi_{n}^{\rm{pts}} < \xi_{n}^{\rm{uni}}$. Therefore, $\delta$-weighted sampling yields a strictly smaller upper bound on MSE than uniform sampling. We further clarify that since Theorem~\ref{thm:rf_mse_bound} is stated as a non-asymptotic result, our claim inherits this property and thus also holds in a non-asymptotic setting. $\blacksquare$

\section{Algorithm for solving MCRM} \label{appendix:complete_CRM_alg}
Algorithm \ref{alg:complete_CRM_alg} provides procedure for applying SCP to solve the MCRM. The formulation of CM and MCRM-ITE are introduced in \cite{zhang2025towards}. All notation used for the \textit{model parameter update} step is defined in Section \ref{subsec:comp_approach}.
\begin{algorithm}
\caption{Iterative approach for solving MCRM}
\label{alg:complete_CRM_alg}
\begin{algorithmic}[1]
    \Statex \hspace*{-1.6em} \textbf{Input:} CM parameter set, MCRM-ITE parameter set, maximum number of iteration: $k^{\text{max}}$
    \Statex \hspace*{-1.6em} \textbf{Output:} Optimized variable values: $\tilde{\bm{v}}_{\text{MCRM}}$ and objective value: $\rm{Obj}_{\text{MCRM}}$
    \State Solve CM \Comment{\textit{initialization}}
    \State Obtain $\bm{v}_{\text{CM}}$ and $\text{Obj}_{\rm{CM}}$ 
    \State Obtain initial variable guess: $\tilde{\bm{v}}_{\text{MCRM-ITE}}[0] \leftarrow \bm{v}_{\text{CM}}$ 
    \State Obtain initial objective value: 
    \Statex \qquad $\text{Obj}_{\rm{MCRM-ITE}}[0] \leftarrow \text{Obj}_{\rm{CM}}$
    \Statex \qquad $\text{Obj}_{\rm{MCRM-APPX}}[0] \leftarrow \text{Obj}_{\rm{CM}}$
    \For{$k=1,...,k^{\text{max}}$}  \Comment{\textit{solving iteratively}}
        \State linearize $\mathcal{C}^{\text{eens}}$ around $\tilde{\bm{v}}_{\text{MCRM-ITE}}[k - 1]$
        \State Solve MCRM-ITE
        \State Obtain $\tilde{\bm{v}}_{\text{MCRM-ITE}}[k]$ and $\text{Obj}_{\text{MCRM-ITE}}[k]$
        \State Use $\tilde{\bm{v}}_{\text{MCRM-ITE}}[k]$ to calculate the updated objective values: $\text{Obj}_{\text{CM}}[k]$, $\text{Obj}_{\text{MCRM}}[k]$ and $\text{Obj}_{\text{MCRM-APPX}}[k]$
        \If {$\text{total variable difference}  \leq \epsilon_1$ \textbf{or} \\ \hspace*{2.1em} $\text{total linearization deviation} \leq \epsilon_2$ \textbf{or} \\
        \hspace*{2.1em} $\text{linearization ratio deviation} \leq \epsilon_3$} 
            \State \textbf{break} 
        \EndIf \Comment{\textit{convergence checked}}
        \If{$v_{i, t}^{[k]} \notin \mathcal{SR}_{i}^{[k]}, \forall i \in \mathcal{N}, t \in \mathcal{T}$}
            \For{each component $i$}
                \State Form $\mathcal{SR}_{i}^{[k]}$
                \State Obtain $\mathcal{V}_{i}^{\text{reg}}$
                \State Fit the regression model of component $i$
                \State Update the model parameter, $\forall t \in \mathcal{T}$
                \State Re-solve MCRM and go to step $8$ 
            \EndFor \Comment{\textit{model parameter updated}}
        \EndIf
    \EndFor
    \State {\textbf{return}} $\tilde{\bm{v}}_{\text{MCRM}}$, $\text{Obj}_{\text{MCRM}}$
\end{algorithmic}
\end{algorithm}

\section{Reliability model covariates selection} \label{appendix:para_values}
Table \ref{tab:para_val} lists all $18$ features in detail, specifying the group to which each feature belongs for every data source $\mathcal{D}$.
\begin{table}
    \centering
        \begin{tabular}{cc}
            \hline
            \textbf{Feature Category} & \textbf{Feature Detail} \\
            \hline
            Heat & $2$m temperature ($K$) \\
                 & $2$m dewpoint temperature ($K$) \\
                 & Surface net solar radiation ($J/m^{2}$) \\
                 & Surface net thermal radiation ($J/m^2$) \\
                 & Surface sensible heat flux ($J/m^2$) \\
            \hline
            Precipitation, & Convective precipitation ($m$) \\
            rain and snow  & Total precipitation ($m$)\\
                 & Convective snowfall ($m$) \\
                 & Snow density ($kg/m^3$) \\
            \hline
            Wind & $10$m u-component of neutral wind ($m/s$) \\
                 & $10$m u-component of wind ($m/s$) \\
                 & $10$m v-component of neutral wind ($m/s$) \\
                 & $10$m v-component of wind ($m/s$) \\
            \hline
            Other & Leaf area index, high vegetation ($m^2/m^2$) \\
                 & Leaf area index, low vegetation ($m^2/m^2$) \\
                 & High vegetation cover \\
                 & Total cloud cover \\
                 & Surface runoff ($m$) \\
            \hline
        \end{tabular}
    \caption{Covariates for the weather data  \cite{copernicus}.}
    \label{tab:para_val}
\end{table}

\section{Partial global prior probability shift} \label{appendix:global_prior_prob_calibration}
The mathematical derivation of this section draws primarily on Section $2.2$ of \cite{saerens2002adjusting}. Our objective is to estimate the real-world conditional failure probability $\text{Pr}_{\text{real}}(y = \omega^j|\bm{x})$ of a component by combining the predictive model trained on our synthetic data and the true prior probabilities $\text{Pr}_{\text{real}}(y = \omega^j)$. We assume that the within class covariate distributions remain unchanged between the synthetic and real life data:
\begin{equation}
    \text{Pr}_{\text{syn}}(\bm{x}|y = \omega^j) = \text{Pr}_{\text{real}}(\bm{x}|y = \omega^j), \forall j = \{-1, 1\}. \label{eq:same_cov_distn}
\end{equation}
By Bayes' theorem and \eqref{eq:same_cov_distn}
\begin{align}
    \text{Pr}_{\text{syn}}(\bm{x}|y = \omega^j) &= \frac{\text{Pr}_{\text{syn}}(y = \omega^j|\bm{x}) \text{Pr}_{\text{syn}}(\bm{x})}{\text{Pr}_{\text{syn}}(y = \omega^j)} \\
    &= \frac{\text{Pr}_{\text{real}}(y = \omega^j|\bm{x}) \text{Pr}_{\text{real}}(\bm{x})}{\text{Pr}_{\text{real}}(y = \omega^j)} \notag \\
    \Rightarrow \text{Pr}_{\text{real}}(y = \omega^j|\bm{x}) &= \frac{\text{Pr}_{\text{syn}}(y = \omega^j|\bm{x})\text{Pr}_{\text{real}}(y = \omega^j)}{\text{Pr}_{\text{syn}}(y = \omega^j)}\frac{\text{Pr}_{\text{syn}}(\bm{x})}{\text{Pr}_{\text{real}}(\bm{x})}\label{eq:bayes_given_x_syn}
\end{align}
In \eqref{eq:bayes_given_x_syn}, we note that $\forall j$ we know the estimate of $\text{Pr}_{\text{syn}}(y = \omega^j|\bm{x})$ $(\hat{\text{P}}\text{r}_{\text{syn}}(y = \omega^j|\bm{x}))$ from the prediction model. $\text{Pr}_{\text{syn}}(y = \omega^j)$ can be directly calculated from our synthetic data and $\text{Pr}_{\text{real}}(y = \omega^j)$ is given. We rewrite \eqref{eq:bayes_given_x_syn} to indicate estimated quantities:
\begin{equation}
    \hat{\text{P}}\text{r}_{\text{real}}(y = \omega^j|\bm{x}) = \frac{\hat{\text{P}}\text{r}_{\text{syn}}(y = \omega^j|\bm{x})\text{Pr}_{\text{real}}(y = \omega^j)}{\text{Pr}_{\text{syn}}(y = \omega^j)}\frac{\text{Pr}_{\text{syn}}(\bm{x})}{\text{Pr}_{\text{real}}(\bm{x})}.\label{eq:bayes_given_x_syn_est}
\end{equation}
By law of total probability $\sum_{j = -1, 1}\hat{\text{P}}\text{r}_{\text{real}}(y = \omega^j|\bm{x}) = 1$ and Bayes' theorem
\begin{align}
     & \left(\frac{\text{Pr}_{\text{syn}}(\bm{x})}{\text{Pr}_{\text{real}}(\bm{x})} \right)^{-1} = \sum_{j = -1, 1} \frac{\text{Pr}_{\text{real}}(\bm{x})}{\text{Pr}_{\text{syn}}(\bm{x})} \hat{\text{P}}\text{r}_{\text{real}}(y = \omega^j|\bm{x}) \notag \\
     = &\sum_{j = -1, 1} \frac{\hat{\text{P}}\text{r}_{\text{real}}(\bm{x}|y = \omega^j)\text{Pr}_{\text{real}}(y = \omega^j)}{\text{Pr}_{\text{real}}(\bm{x})} \frac{\text{Pr}_{\text{real}}(\bm{x})}{\text{Pr}_{\text{syn}}(\bm{x})} \notag \\
     = &\sum_{j = -1, 1} \frac{\hat{\text{P}}\text{r}_{\text{syn}}(\bm{x}|y = \omega^j)\text{Pr}_{\text{real}}(y = \omega^j)}{\text{Pr}_{\text{syn}}(\bm{x})} \notag \\
     = &\sum_{j = -1, 1} \frac{\hat{\text{P}}\text{r}_{\text{syn}}(y = \omega^j|\bm{x})\text{Pr}_{\text{syn}}(\bm{x})}{\text{Pr}_{\text{syn}}(y = \omega^j)} \frac{\text{Pr}_{\text{real}}(y = \omega^j)}{\text{Pr}_{\text{syn}}(\bm{x})} \notag \\
     = &\sum_{j = -1, 1} \hat{\text{P}}\text{r}_{\text{syn}}(y = \omega^j|\bm{x}) \frac{\text{Pr}_{\text{real}}(y = \omega^j)}{\text{Pr}_{\text{syn}}(y = \omega^j)}
     \label{eq:quantity_to_cancel}
\end{align}
By \eqref{eq:bayes_given_x_syn_est} and \eqref{eq:quantity_to_cancel}, we obtain $\forall \diamond \in \{\text{b}, \text{l}\}$
\begin{align*}
    &\hat{\text{P}}\text{r}_{\text{real}, i}(y_i^{\diamond} = \omega_i^{-1}|\bm{x}_i^{\diamond}) \notag \\
    = &\frac{\frac{\hat{\text{P}}\text{r}_{\text{syn}}(y_i^{\diamond} = \omega_i^{-1}|\bm{x}_i^{\diamond})\text{Pr}_{\text{real}}(y^{\diamond} = \omega^{-1})}{\text{Pr}_{\text{syn}}(y^{\diamond} = \omega^{-1})}}{\sum_{j = -1, 1} \hat{\text{P}}\text{r}_{\text{syn}}(y_i^{\diamond} = \omega^{j}|\bm{x}_i^{\diamond}) \frac{\text{Pr}_{\text{real}}(y^{\diamond} = \omega^j)}{\text{Pr}_{\text{syn}}(y^{\diamond} = \omega^j)}} \notag \\
    = &\frac{\hat{\text{P}}\text{r}_i^{\diamond, \text{raw}}\text{Pr}_i^{\diamond, \text{avg}}\frac{n^{\diamond}_{\text{syn}}}{n_{\text{syn,fail}}}}{\hat{\text{P}}\text{r}_i^{\diamond, \text{raw}}\text{Pr}_i^{\diamond, \text{avg}} \frac{n_{\text{syn}}}{n^{\diamond}_{\text{syn,fail}}} + (1 - \hat{\text{P}}\text{r}_i^{\diamond, \text{raw}})(1 - \text{Pr}_i^{\diamond, \text{avg}}) \frac{n_{\text{syn}}}{n_{\text{syn}} - n^{\diamond}_{\text{syn,fail}}}}. \blacksquare
\end{align*}
Global prior calibration uniformly rescales all probabilities, but doing so in the extreme tails is unreliable because these regions contain few validation examples. As a result, the rescaling factor is dominated by the prior and can overwhelm high confidence predictions \cite{niculescu2005predicting}. Literature exists to show formally that calibrating only in a neighborhood of the decision threshold is sufficient for Bayes risk consistency and probabilities that are already close to $0$ or $1$ do not affect decisions or expected loss. Therefore, we limit the prior correction to the central interval s.t. $\mathcal{CI} = [0.05, 0.95]$ \cite{sahoo2021reliable}.

\section{IEEE 33 bus test system layout} \label{appendix:sys_layout}
\begin{enumerate}
    \item \textbf{DGs}: Bus $15$, $16$, $17$, $18$, $21$, $23$, $24$, $26$, $30$
    \item \textbf{PVs}:
        \begin{itemize}
            \item \textit{Low PV case}: Bus $17$, $18$, $23$, $24$
            \item \textit{High PV case}: Bus $15$, $16$, $17$, $18$, $21$, $23$, $24$, $26$, $30$
        \end{itemize}
    \item \textbf{BESSs}: Bus $17$, $18$, $23$, $24$
    \item \textbf{DR}: Bus $2$, $3$, $5$
\end{enumerate}
Figure \ref{fig:33_bus_power_radial_system} shows the system topology and locations of DERs.

\begin{figure}
    \centering
    \resizebox{\linewidth}{!}{  \begin{tikzpicture}[>={Latex[width=2mm,length=2mm]},
                      substation/.style={circle, draw=blue!55, fill = blue!50, text = white, text centered, minimum size=1cm, inner sep=2pt},
                      bus/.style={circle, draw=darkgreen!55, fill = darkgreen!50, text = white, text centered, minimum size=0.75cm, inner sep=2pt},
                      unit/.style={align=center, font=\scriptsize, execute at begin node=\setlength{\baselineskip}{10pt}},
                      scale=0.5, every node/.style={scale=0.5}]

  \node[substation] (bus0) { Sub };
  \node[bus, right=0.4cm of bus0]  (bus1)  {1};
  \node[bus, right=0.4cm of bus1]  (bus2)  {2};
  \node[bus, right=0.4cm of bus2]  (bus3)  {3};
  \node[bus, right=0.4cm of bus3]  (bus4)  {4};
  \node[bus, right=0.4cm of bus4]  (bus5)  {5};
  \node[bus, right=0.4cm of bus5]  (bus6)  {6};
  \node[bus, right=0.6cm of bus6] (bus14) {14};
  \node[bus, right=0.4cm of bus14] (bus15) {15};
  \node[bus, right=0.4cm of bus15] (bus16) {16};
  \node[bus, right=0.4cm of bus16] (bus17) {17};

  \node[bus, above=0.4cm of bus3] (bus22) {22};
  \node[bus, right=0.4cm of bus22] (bus23) {23};
  \node[bus, right=0.4cm of bus23] (bus24) {24};

  \node[bus, below=0.4cm of bus6] (bus25) {25};
  \node[bus, right=0.4cm of bus25] (bus26) {26};
  \node[bus, right=0.6cm of bus26] (bus30) {30};
  \node[bus, right=0.4cm of bus30] (bus31) {31};
  \node[bus, right=0.4cm of bus31] (bus32) {32};

  \node[bus, below=0.8cm of bus2] (bus18) {18};
  \node[bus, right=0.4cm of bus18] (bus19) {19};
  \node[bus, right=0.4cm of bus19] (bus20) {20};
  \node[bus, right=0.4cm of bus20] (bus21) {21};

  \draw[->, pink!86] (bus0) -- (bus1);
  \draw[->, pink!86] (bus1) -- (bus2);
  \draw[->, pink!86] (bus2) -- (bus3);
  \draw[->, pink!86] (bus3) -- (bus4);
  \draw[->, pink!86] (bus4) -- (bus5);
  \draw[->, pink!86] (bus5) -- (bus6);
  \draw[dotted, thick, ->, pink!86] (bus6) -- (bus14);
  \draw[->, pink!86] (bus14) -- (bus15);
  \draw[->, pink!86] (bus15) -- (bus16);
  \draw[->, pink!86] (bus16) -- (bus17);

  \draw[->, pink!86] (bus2) |- (bus22);
  \draw[->, pink!86] (bus22) -- (bus23);
  \draw[->, pink!86] (bus23) -- (bus24);

  \draw[->, pink!86] (bus5) |- (bus25);
  \draw[->, pink!86] (bus25) -- (bus26);
  \draw[dotted, thick, ->, pink!86] (bus26) -- (bus30);
  \draw[->, pink!86] (bus30) -- (bus31);
  \draw[->, pink!86] (bus31) -- (bus32);

  \draw[->, pink!86] (bus1) |- (bus18);
  \draw[->, pink!86] (bus18) -- (bus19);
  \draw[->, pink!86] (bus19) -- (bus20);
  \draw[->, pink!66] (bus20) -- (bus21);

  \node[unit, below=0.01cm of bus2]{\textcolor{purple!90}{DR}};
  \node[unit, below=0.01cm of bus3]{\textcolor{purple!90}{DR}};
  \node[unit, above=0.01cm of bus5]{\textcolor{purple!90}{DR}};
  \node[unit, above=0.01cm of bus15]{\textcolor{purple!90}{DG}\\ \textcolor{purple!60}{PV}};
  \node[unit, above=0.01cm of bus16]{\textcolor{purple!90}{DG}\\ \textcolor{purple!60}{PV}};
  \node[unit, above=0.01cm of bus17]{\textcolor{purple!90}{DG}\\\textcolor{purple!90}{PV}\\\textcolor{purple!90}{Battery}};
  \node[unit, below=0.01cm of bus18]{\textcolor{purple!90}{DG}\\\textcolor{purple!90}{PV}\\\textcolor{purple!90}{Battery}};
  \node[unit, below=0.01cm of bus21]{\textcolor{purple!90}{DG}\\ \textcolor{purple!60}{PV}};
  \node[unit, above=0.01cm of bus23]{\textcolor{purple!90}{DG}\\\textcolor{purple!90}{PV}\\\textcolor{purple!90}{Battery}};
  \node[unit, above=0.01cm of bus24]{\textcolor{purple!90}{DG}\\\textcolor{purple!90}{PV}\\\textcolor{purple!90}{Battery}};
  \node[unit, below=0.01cm of bus26]{\textcolor{purple!90}{DG}\\ \textcolor{purple!60}{PV}};
  \node[unit, below=0.01cm of bus30]{\textcolor{purple!90}{DG}\\ \textcolor{purple!60}{PV}};
\end{tikzpicture}}
    \caption{33-bus power system radial network topology.}
    \label{fig:33_bus_power_radial_system}
\end{figure}
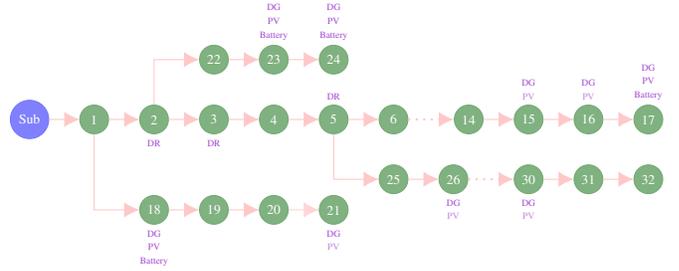
\FloatBarrier

\bibliographystyle{IEEEtran}
\bibliography{IEEE_Reliability/ref}

\end{document}